\input ./style/arxiv-general.cfg
\documentclass[aap,MSNbibl,nameyear,dvips]{arximspdf}
\makeatletter
   \@ifpackageloaded{graphicx}{}{\usepackage{graphicx}}
\makeatother
\usepackage{dcolumn}

\doi{10.1214/14-AAP1072}
\volume{25}
\issue{6}
\pubyear{2015}
\firstpage{3209}
\lastpage{3250}
\docsubty{FLA}

\makeatletter
\newcolumntype{d}[1]{D{.}{.}{#1}}
\newcommand{\rrvert}{\vert}
\newcommand{\llvert}{\vert}
\newcommand{\eqref}[1]{(\ref{#1})}
\newtheorem{theorem}{Theorem}
\newproclaim{ass}{Assumption}
\newproclaim{algorithm}{Algorithm}
\newtheorem{lemma}{Lemma}
\newtheorem{proposition}{Proposition}
\newproclaim{remark}{Remark}

\makeatother

\begin{document}
\begin{frontmatter}

\title{Steady-state simulation of reflected Brownian motion and related
stochastic networks\thanksref{T1}}
\runtitle{Steady-state simulation of reflected Brownian motion}

\thankstext{T1}{Supported in part by the Grants NSF-CMMI-0846816
and NSF-CMMI-1069064.}

\begin{aug}
\author[A]{\fnms{Jose} \snm{Blanchet}\corref{}\ead[label=e1]{jose.blanchet@columbia.edu}}
\and
\author[B]{\fnms{Xinyun} \snm{Chen}\ead[label=e2]{xinyun.chen@stonybrook.edu}}
\runauthor{J. Blanchet and X. Chen}
\affiliation{Columbia University and Stony Brook University}
\address[A]{Industrial Engineering\\
\quad and Operations Research\\
Columbia University\\
340 S. W. Mudd Building\\
500 W. 120 Street\\
New York, New York 10027\\
USA\\
\printead{e1}}
\address[B]{Department of Applied Mathematics\\
\quad and Statistics\\
Stony Brook University\\
Math Tower B148\\
Stony Brook, New York 11794-3600\\
USA\\
\printead{e2}}
\end{aug}
%
%
\received{\smonth{1} \syear{2012}}
%
\revised{\smonth{9} \syear{2014}}

%
\begin{abstract}
This paper develops the first class of algorithms that enable unbiased
estimation of steady-state expectations for multidimensional reflected
Brownian motion. In order to explain our ideas, we first consider the
case of compound Poisson (possibly Markov modulated) input. In this
case, we analyze the complexity of our
procedure as the dimension of the network increases and show that, under
certain assumptions, the algorithm has polynomial-expected termination
time. Our methodology includes procedures that are of
interest beyond steady-state simulation and reflected processes. For instance,
we use wavelets to construct a piecewise linear function that can be
guaranteed to be within $\varepsilon$ distance (deterministic) in the uniform
norm to Brownian motion in any compact time interval.
\end{abstract}

%
\end{frontmatter}

\section{Introduction}\label{sec1}

This paper studies simulation methodology that allows estimation, without
any bias, of steady-state expectations of multidimensional reflected
processes. Our algorithms are presented with companion rates of
\mbox{convergence}. Multidimensional reflected processes, as we shall explain,
are very
important for the analysis of stochastic queueing networks. However, in
order to
motivate the models that we study, let us quickly review a formulation
introduced by \citet{Kella1996}.

Consider a network of $d$ queueing stations indexed by $\{1,2,\ldots,d\}$.
Suppose that jobs arrive to the network according to a Poisson process with
rate $\lambda$, denoted by $(N ( t ) \dvtx t\geq0)$. Specifically,
the $%
k $th arrival brings a vector of job requirements $\mathbf{W} (
k ) = ( W_{1} ( k ),\ldots,W_{d} ( k )
 ) ^{T}$
which are nonnegative random variables (r.v.'s), and they add to the
workload at each station right at the moment of arrival. So if the $k$th
arrival occurs at time $t$, the workload of the $i$th station (for $%
i\in\{1,\ldots,d\}$) increases by $W_{i} ( k ) $ units right
at time $%
t $. We assume that $\mathbf{W}=(\mathbf{W} ( k ) \dvtx k\geq1$) is a
sequence of i.i.d. (independent and identically distributed) nonnegative
r.v.'s. For fixed $k$, the coordinates of $\mathbf{W} ( k ) $ are
not necessarily independent; however, $\mathbf{W}$ is assumed to be
independent of $N ( \cdot ) $.

Throughout the paper we shall use boldface to write vector quantities, which
are encoded as columns. For instance, we write $\mathbf{y}= (
y_{1},\ldots,y_{d} ) ^{T}$.

The total amount of external work that arrives to the $i$th station up
to (and
including) time $t$ is denoted by
\[
J_{i} ( t ) =\sum_{k=1}^{N ( t ) }W_{i}
( k ).
\]

Now, assume that the workload at the $i$th station is processed as a fluid
by the server at a rate $r_{i}$, continuously in time. This means that if
the workload in the~$i$th station remains strictly positive during the time
interval $[t,t+dt]$, then the output from station $i$ during this time
interval equals $r_{i}\,dt$. In addition, suppose that a proportion $%
Q_{i,j}\geq0$ of the fluid processed by the $i$th station is circulated to
the $j$th server. We have that $\sum_{j=1}^{d}Q_{i,j}\leq1$, $Q_{i,i}=0$,
and we define $Q_{i,0}=1-\sum_{j=1}^{d}Q_{i,j}$. The proportion $Q_{i,0}$
corresponds to the fluid that goes out of the network from station $i$.

The dynamics stated in the previous paragraph are expressed formally by a
differential equation as follows. Let $Y_{i} ( t ) $ denote the
workload content of the $i$th station at time $t$. Then for given
$Y_{i} (
0 ) $, we have
%
\begin{eqnarray}
\label{S1b} dY_{i} ( t ) & =&dJ_{i} ( t ) -
r_{i}I \bigl( Y_{i} ( t ) >0 \bigr) \,dt+\sum
_{j:j\neq i}Q_{j,i}r_{j}I \bigl(
Y_{j} ( t ) >0 \bigr) \,dt
\nonumber
\\
& =&dJ_{i} ( t ) -r_{i}\,dt+\sum
_{j:j\neq i}Q_{j,i}r_{j}\,dt
\\
&&{} +r_{i}I \bigl( Y_{i} ( t ) =0 \bigr) \,dt-\sum
_{j:j\neq
i}Q_{j,i}r_{j}I \bigl(
Y_{j} ( t ) =0 \bigr) \,dt\nonumber
\end{eqnarray}
for $i\in\{1,\ldots,d\}$. It is well known that the resulting vector-valued
workload process, $\mathbf{Y} ( t ) = ( Y_{1} ( t ),\ldots,Y_{d} ( t )  ) ^{T}$, is Markovian. The differential
equation (\ref{S1b}) admits a unique piecewise linear solution that is
right-continuous and has left limits (RCLL). This can be established by
elementary methods, and we shall comment on far-reaching extensions shortly.

The equations given in (\ref{S1b}) take a neat form in matrix notation. This
notation is convenient when examing stability issues and other topics which
are related to the steady-state simulation problem we investigate.
In particular, let $\mathbf{r}= ( r_{1},\ldots,r_{d} ) ^{T}$ be
the column vector corresponding to the service rates, write $R= (
I-Q ) ^{T}$ and define
\[
\mathbf{X} ( t ) =\mathbf{J} ( t ) -R\mathbf{r}t,
\]
where $\mathbf{J} ( t ) $ is a column vector with its $i$th
coordinate equal to $J_{i} ( t ) $. Then
equation~(\ref{S1b}) can be expressed as
%
\begin{equation}
\mathbf{Y} ( t ) =\mathbf{Y} ( 0 ) +\mathbf{X} ( t ) +R\mathbf{L} ( t ),
\label{SP1}
\end{equation}
where $\mathbf{L} ( t ) $ is a column vector with its $i$th
coordinate equal to
\[
L_{i} ( t ) =\int_{0}^{t}r_{i}I
\bigl( Y_{i} ( s ) =0 \bigr) \,ds.
\]

As mentioned earlier, $\mathbf{Y=(Y} ( t ) \dvtx t\geq0)$ is a Markov
process. Let us assume that $Q^{n}\rightarrow0$ as $%
n\rightarrow\infty$. This assumption is synonymous with the assumption
that the network
is open. In detail, for each $i$ such that $\lambda_{i}>0$, there
exists a path ($i_{1},i_{2},\ldots,i_{k})$
satisfying that $\lambda
_{i}Q_{i,i_{1}}Q_{i_{1},i_{2}}\cdots Q_{i_{k-1},i_{k}}>0$ with $i_{k}=0$
and $k\leq d$. In addition, under this assumption the matrix $R^{-1}$ exists and has
nonnegative coordinates. To ensure stability, we assume that
$R^{-1}E\mathbf{X} (1 ) <0$---inequalities involving vectors are understood
coordinate-wise throughout the paper. It follows from Theorem~2.4 of %
\citet{KellaRamasubramanian2012} that $\mathbf{Y} ( t )$
converges in distribution to $\mathbf{Y}%
 ( \infty ) $ as $t\rightarrow\infty$, where $\mathbf{Y} (
\infty ) $ is an r.v. with the (unique) stationary distribution of
$%
\mathbf{Y} ( \cdot ) $.

The first contribution of this paper is that we develop an exact sampling
algorithm (i.e., simulation without bias) for $\mathbf{Y} ( \infty
 ) $. This algorithm is developed in Section~\ref{SecCompPoi}
of this paper under the assumption that $\mathbf{W} ( k ) $
has a finite moment-generating function. In addition, we analyze the
order of computational complexity
(measured in terms of expected random numbers generated) of our
algorithm as $d$ increases, and we show that it is
polynomially bounded.

Moreover, we extend our exact sampling algorithm to the case in
which there is an independent Markov chain driving the arrival rates, the
service rates, and the distribution of job sizes at the time of arrivals.
This extension is discussed in Section~\ref{SectionExtension}.

The workload process $ ( \mathbf{Y} ( t ) \dvtx t\geq0 )
$ is
a particular case of a reflected (or constrained) stochastic network.
Although the models introduced in the previous paragraphs are
interesting in
their own right, our main interest is the steady-state
simulation techniques for reflected Brownian motion. These techniques
are obtained by
abstracting the construction formulated in (\ref{SP1}). This
abstraction is
presented in terms of a Skorokhod problem, which we describe as
follows. Let $\mathbf{X=} ( \mathbf{X}%
 ( t ) \dvtx t\geq0 ) $ with $\mathbf{X}(0)\geq0$, and $R$
be an $M$-matrix $R$ so that the inverse $R^{-1}$ exists and has
nonnegative coordinates. To solve the Skorokhod problem requires
finding a pair of processes $ (
\mathbf{Y,L} ) $ satisfying equation (\ref{SP1}), subject to:
\begin{longlist}[(iii)]
\item[(i)] $\mathbf{Y} ( t ) \geq0$ for each $t$,

\item[(ii)] $L_{i} ( \cdot ) $ nondecreasing for each $i\in\{
1,\ldots, d\}$ and
$L_{i} ( 0 ) =0$,

\item[(iii)] $\int_{0}^{t}Y_{i} ( s ) \,dL_{i} ( s ) =0$ for
each $t$.
\end{longlist}

Eventually we shall take the input process $\mathbf{X} ( \cdot
) $ as a
Brownian motion with constant drift $\mathbf{v}=E\mathbf{X} (
1 ) $
and nondegenerate covariance matrix $\Sigma$. There
then exists a strong solution (i.e., path-by-path and not only in law)
to the stochastic differential equation (SDE) (\ref{SP1}) subject to the
Skorokhod problem constraints (i) to (iii), and the initial condition
$\mathbf{%
Y} ( 0 ) $. This was proved by \citet{HarrisonReiman1981}, who
introduced the notion of reflected Brownian motion (RBM). When $R$ is
an $M$-matrix, $R^{-1}\bolds{\mu}<0$ is a necessary and
sufficient condition for the stability of an RBM; see \citet
{HarrisonWilliams1987}.
Our algorithm for the RBM is motivated by the
fact that in great generality (i.e., only requiring the existence of
variances of
service times and inter-arrival times), the so-called generalized Jackson
networks (which are single-server queues connected with Markovian routing)
converge weakly to a reflected Brownian motion in a heavy traffic asymptotic
environment as in \citet{Reiman1984}. Moreover, recent papers from
\citet{GamarnikZeevi2006} and \citet{BudhirajaLee2009} have
shown that convergence occurs also at the level of steady-state
distributions. Therefore, reflected Brownian motion
(RBM) plays a central role in queueing theory.


%

The second contribution of this paper is the development of an
algorithm that allows estimation with no bias of $E [g ( \mathbf
{Y} (
\infty )  ) ] $ for positive and continuous functions $%
g ( \cdot ) $. Moreover, given $\varepsilon>0$,
we provide a simulation algorithm that outputs a random variable
$\mathbf{Y}%
_{\varepsilon} ( \infty ) $ that can be guaranteed to be
within $\varepsilon$ distance (say in the Euclidian norm) from an
unbiased sample $\mathbf{Y} ( \infty ) $ from the
steady-state distribution of RBM. This contribution is developed in
Section~\ref{SecRBM} of this paper. We show that the number of Gaussian random
variables generated to produce $\mathbf{Y}_{\varepsilon} ( \infty
 ) $ is of order $O(\varepsilon^{-a_{C}-2}\log(1/\varepsilon))$
as $\varepsilon\searrow0$, where $a_{C}$ is a
constant only depending on the covariance matrix of the Brownian
motion; see
Section~\ref{SubSecCCRBM}. In the special case when the $d$-dimensional
Brownian motion has nonnegative correlations, the number of random
variables generated is of order $O(\varepsilon^{-d-2}\log
(1/\varepsilon
)) $.

Our methods allow estimation without bias of $E [g ( \mathbf
{Y} (
t_{1} ),\mathbf{Y} ( t_{2} ),\ldots,\break \mathbf{Y} (
t_{m} )  )  ]$ for a positive function $g ( \cdot
 ) $
continuous almost everywhere and for any $0<t_{1}<t_{2}<\cdots<t_{m}$.
Simulation of RBM has been studied in the literature. In the one-dimensional
setting it is not difficult to sample RBM exactly; this follows, for
instance, from the methods in \citet{Devroye2009}. The paper of %
\citet{Asmussenetal1995} also studies the one-dimensional case and
provides an enhanced Euler-type scheme with an improved convergence rate.
The work of \citet{BurdzyChen2008} provides approximations of reflected
Brownian motion with orthogonal reflection (the case in which $R=I$).

With regard to steady-state computations, the work of %
\citet{DaiHarrison1992} provides numerical methods for approximating the
steady-state expectation by numerically evaluating the density of
$\mathbf{Y}%
 ( \infty ) $. In contrast to our methods, Dai and Harrison's
procedure is based on projections in mean-squared norm with respect to a
suitable reference measure. Since such an algorithm is nonrandomized,
it is
therefore, in some sense, preferable to simulation approaches, which are
necessarily randomized. However, the theoretical justification of Dai and
Harrison's algorithm relies on a conjecture that is believed to be true but
has not been rigorously established; see \citet{DaiDieker2011}. In addition,
no rate of convergence is known for this procedure, even assuming that the
conjecture is true.

Finally, we briefly discuss some features of our procedure and our strategy
at a high level. There are two sources of bias that arise in the
setting of
steady-state simulation of RBM. First, discretization error in the
simulation of the process $\mathbf{Y}$ is inevitable due to the
continuous nature of Brownian motion, especially when the reflection
matrix $R$ is not the identity. This issue is present even in finite
time horizon. The
second issue is, naturally, that we are concerned with steady-state
expectations which inherently involve, in principle, an infinite time
horizon.

In order to concentrate on removing the bias issues arising from the
infinite horizon, we first consider the reflected compound Poisson case
where we can simulate the solution of the
Skorokhod problem in any finite interval exactly and without any bias.
Our strategy is based on the dominated coupling from the past
(DCFTP). This technique was proposed by \citet{Kendall2004}, following the
introduction of coupling from the past by \citet{ProppWilson1996}. The
idea behind DCFTP is to construct suitable upper- and lower-bound
processes that can be simulated in stationarity and backward in time.
We take the lower bound to be the process identically equal to zero. We use
results from \citet{HarrisonWilliams1987} (for the RBM) and \citet
{Kella1996} (for the reflected compound Poisson process), to construct
an upper bound process based on the solution of the Skorokhod
problem with reflection matrix $R=I$. It turns out that simulation of the
stationary upper-bound process backward involves sampling the infinite
horizon maximum (coordinate-wise) from $t$ to infinity of a $d$-dimensional compound Poisson Process with negative drift. We use sequential
acceptance/rejection techniques (based on a exponential tilting
distributions used in rare-event simulation) to simulate from an infinite
horizon maximum process.

Then we turn to RBM. A problem that arises, in addition to the
discretization error given the continuous nature of Brownian motion, is the
fact that in dimensions higher than one (as in our setting) RBM never
reaches the origin. Nevertheless, it will be arbitrarily close to the origin,
and we shall certainly leverage off this property to obtain simulation that
is guaranteed to be $\varepsilon$-close to a genuine steady-state sample. Now
in order to deal with the discretization error we use wavelet-based
techniques. We take advantage of a well-known wavelet construction of
Brownian motion; see \citet{Steele2001}.

Instead of simply simulating Brownian motion using the wavelets,
which is the standard practice, we simulate the wavelet coefficients jointly
with suitably defined random times. Consequently, we are able to guarantee
with probability one that our wavelet approximation is $\varepsilon
$-close in
the uniform metric to Brownian motion in any compact time interval (note
that $\varepsilon$ is deterministic and defined by the user; see Section~\ref%
{SubWave}).

Finally, we use the following fact. Let process $\mathbf{Y}$ be the
solution to the Skorokhod problem. Then the process $\mathbf{Y}$, as a
function of the input process $%
\mathbf{X}$, is Lipschitz continuous with a computable Lipschitz constant,
under the uniform topology. These observations combined with an
additional randomization, in the spirit of \citet{Beskosetal2012}, allow
estimation with no bias of the steady-state expectation.

We strongly believe that the use of tolerance-enforced coupling based on
wavelet constructions, as we illustrate here, can be extended more broadly
in the numerical analysis of the Skorokhod and related problems.

We perform some numerical experiments to validate
our algorithms. Our results are reported in Section~\ref%
{SectionNumerics}. Further numerical experiments are pursued in a companion
paper, in which we also discuss further implementation issues and some
adaptations, which are specially important in the case of RBM.

The rest of the paper is organized as follows: in Section~\ref{SecCompPoi},
we consider the problem of exact simulation from the steady-state
distribution of the reflected compound Poisson process discussed
earlier; we then
show how our procedure is adapted without major complications to
Markov-modulated input in Section~\ref{SectionExtension}; in
Section~\ref{SecRBM}, we continue explaining the main strategy to be
used for the
reflected Brownian motion case; finally, the numerical
experiments are given in Section~\ref{SectionNumerics}.

\section{Exact simulation of reflected compound Poisson processes}\label%
{SecCompPoi}

The model that we consider has been explained at the beginning of the
\hyperref[sec1]{Introduction}. We summarize the assumptions that we shall impose next.

\emph{Assumptions}:

(A1) the matrix $R$ is an $M$-matrix;

(A2) $R^{-1}E\mathbf{X} ( 1 ) <0$ (recall that inequalities apply
coordinate-wise for vectors);

(A3) there exists $\bolds{\theta}>0$, $\bolds{\theta}\in
\mathbb{R}^{d}$ such that
\[
E\bigl[\exp \bigl( \bolds{\theta}^{T}\mathbf{W} ( k ) \bigr)
\bigr] <\infty.
\]

We have commented on (A1) and (A2) in the \hyperref[sec1]{Introduction}. Assumption (A3) is
important in order to do exponential tilting when we simulate a
stationary version of the upper-bound process.

In addition to (A1) to (A3), we shall assume that one can simulate from
exponential tilting distributions associated to the marginal
distribution of
$\mathbf{W} ( k ) $. That is, we can simulate from $%
P_{\theta_{i}} ( \cdot ) $ such that
\begin{eqnarray*}
&& P_{\theta_{i}} \bigl( W_{1} ( k ) \in dy_{1},
\ldots,W_{d} ( k ) \in dy_{d} \bigr)
\\
&&\qquad=  \frac{\exp ( \theta_{i}y_{i} ) }{E\exp (
\theta_{i}W_{i} ( k )  ) }P \bigl( W_{1} ( k ) \in dy_{1},
\ldots,W_{d} ( k ) \in dy_{d} \bigr),
\end{eqnarray*}
where $\theta_{i}\in\mathbb{R}$ and $E\exp ( \theta_{i}W_{i} (
k )  ) <\infty$. We will determine the value of $\theta_i$ through
assumption (A3b), as given below.

Let us briefly explain our program, which is based on DCFTP. First, we will
construct a \textit{stationary} dominating process $ ( \mathbf{Y}%
^{+} ( s ) \dvtx -\infty<s\leq0 ) $ that is \emph{coupled}
with our
target process, that is, a stationary version of the process $ (
\mathbf{Y} ( s ) \dvtx -\infty<s\leq0 ) $ satisfying the Skorokhod
problem (\ref{SP1}). Under coupling, the dominating process satisfies
%
\begin{equation}
R^{-1}\mathbf{Y} ( s ) \leq R^{-1}\mathbf{Y}^{+}
( s ), \label{Dom}
\end{equation}
for each $s\leq0$. We then simulate the process $\mathbf{Y}^{+} (
\cdot ) $ backward up to a time $-\tau\leq0$ such that $\mathbf{Y}%
^{+} ( -\tau ) =0$. Following the tradition of the CFTP literature,
we call a time $-\tau$ such that $\mathbf{Y}^{+} ( -\tau ) =0$ a
coalescence time. Since $\mathbf{Y} ( s ) \geq0$, inequality
(\ref%
{Dom}) yields $\mathbf{Y} ( -\tau ) =0$. The next and final
step in our strategy is to evolve the solution $\mathbf{Y}(s)$ of the
Skorokhod problem (\ref{SP1}) forward from $s=-\tau$ to $s=0$ with
$\mathbf{%
Y}(-\tau)=0$, \textit{using the same input that drives the
construction of}
$\mathbf{(Y}^+ ( s ) \dvtx -\tau\leq s\leq0)$ so that $\mathbf{Y}$
and $%
\mathbf{Y}^+$ are coupled. The output is therefore $\mathbf{Y} (
0 ) $, which is stationary. The precise algorithm will be
summarized in Section~\ref{Subsec_Main1}.

So, a crucial part of the whole plan is the construction of $\mathbf{Y}%
^{+} ( \cdot ) $ together with a coupling that guarantees
inequality (\ref{Dom}). In addition, the coupling must be such that one can
use the driving randomness that defines $\mathbf{Y}^{+} ( \cdot
 ) $
directly as an input to the Skorokhod problem (\ref{SP1}) that is then used
to evolve $\mathbf{Y}^{+} ( \cdot ) $. We shall first start by
constructing a time reversed stationary version of a suitable dominating
process~$\mathbf{Y}^{+}$.

\subsection{Construction of the dominating process}\label{sec2.1}

In order to construct the dominating process $\mathbf{Y}^{+} (
\cdot )$, we first need the following result attributed to %
\citet{Kella1996} (Lemma~3.1).

\begin{lemma}
\label{LmK_W_Comp}There exists $\mathbf{z}$ such that $E\mathbf{X} (
1 ) <\mathbf{z}$ and $R^{-1}\mathbf{z}<0$. Moreover, if
\[
\mathbf{Z} ( t ) =\mathbf{X} ( t ) -\mathbf{z}t,
\]
and $\mathbf{Y}^{+} ( \cdot ) $ is the solution to the
Skorokhod problem
%
\begin{eqnarray}\label{SP_Bnd}
d\mathbf{Y}^{+} ( t ) & =&d\mathbf{Z} ( t ) +d\mathbf
{L}%
^{+} ( t ),\qquad \mathbf{Y}^{+} ( 0 ) =
\mathbf{y}_{0},
\nonumber
\\[-8pt]
\\[-8pt]
\nonumber
\mathbf{Y}^{+} ( t ) & \geq&0,\qquad Y_{j}^{+}
( t ) \,dL_{j}^{+} ( t )=0,\qquad L _{j}^{+}
( 0 ) =0, %
\qquad dL_{j}^{+} ( t ) \geq0,
\end{eqnarray}
then $0\leq R^{-1}\mathbf{Y} ( t ) \leq R^{-1}\mathbf
{Y}^{+} (
t ) $ for all $t\geq0$ where $\mathbf{Y} ( \cdot ) $ solves
the Skorokhod problem
\begin{eqnarray*}
d\mathbf{Y} ( t ) & =&d\mathbf{X} ( t ) +R\,d\mathbf {L} ( t ),\qquad \mathbf{Y} (
0 ) =\mathbf{y}_{0},
\\
\mathbf{Y} ( t ) & \geq& 0,\qquad Y_{j} ( t ) \,dL_{j} ( t )
=0,\qquad L_{j} ( 0 ) =0,\qquad dL_{j} ( t ) \geq0.
\end{eqnarray*}
\end{lemma}

We note that computing $\mathbf{z}$ from the previous lemma is not
difficult. One can simply pick $\mathbf{z}=E\mathbf{X} ( 1 )
+\delta\mathbf{1}$, where $\mathbf{1}= ( 1,\ldots,1 ) ^{T}$
and with $\delta$
chosen so that $0<\delta R^{-1}\mathbf{1}<-R^{-1}E\mathbf{X} (
1 ) $. In what follows we shall assume that $\mathbf{z}$ has been selected in
this form, and we shall assume without loss of generality that
$E[\mathbf{Z}%
(1)]<0$.

The Skorokhod problem corresponding to the dominating process can be solved
explicitly. It is not difficult to verify [see, e.g., %
\citet{HarrisonReiman1981}] that if $\mathbf{Y}^{+} ( 0 ) =0$,
the solution of the Skorokhod problem (\ref{SP_Bnd}) is given by
%
\begin{equation}
\mathbf{Y}^{+} ( t ) =\mathbf{Z} ( t ) -\min_{0\leq
u\leq
t}
\mathbf{Z} ( u ) =\max_{0\leq u\leq t}\bigl(\mathbf{Z} ( t ) -%
\mathbf{Z} ( u ) \bigr), \label{SKR_solution}
\end{equation}
where the running maximum is obtained coordinate-by-coordinate.

In order to construct a stationary version of $\mathbf{Y}^{+} (
\cdot ) $ backward in time, we first extend $\mathbf{Z} (
\cdot ) $ to a two-sided compound Poisson process with $\mathbf{Z}%
 ( 0 ) =0$. We define a time-reversal of $\mathbf{Z}(\cdot)$
as $%
\mathbf{Z}^{\leftarrow} ( t ) =-\mathbf{Z} ( -t ) $.
It is
easy to check that $\mathbf{Z}^{\leftarrow} ( \cdot ) $ has
stationary and independent increments that are identically distributed as
those of~$\mathbf{Z}(\cdot)$.

For any given $T\leq0$, we define a process $\mathbf{Z}_{T}^{\leftarrow}$
via $\mathbf{Z}_{T}^{\leftarrow} ( t ) =\mathbf{Z}%
^{\leftarrow} ( T+t ) $ for $0\leq t\leq|T|$. And for any
given $%
\mathbf{y}\geq0$ we define $\mathbf{Y}_{T}^{+} ( t,\mathbf{y} )$
for $0\leq t\leq|T|$ to be the solution to the Skorokhod problem with input
process $\mathbf{Z}_T^{\leftarrow}$, initial condition $\mathbf
{Y}_T^+(0,%
\mathbf{y})=\mathbf{y}$ and reflection matrix $R=I$. In detail, $\mathbf
{Y}%
_{T}^{+} ( \cdot,\mathbf{y} )$ solves
%
\begin{eqnarray}\label{Ypls_def}
d\mathbf{Y}_{T}^{+} ( t,\mathbf{y} ) &=&d
\mathbf{Z}%
_{T}^{\leftarrow} ( t ) +d\mathbf{L}_{T}^{+}
( t,\mathbf{y} 
 ),\qquad \mathbf{Y}_{T}^{+}( 0,\mathbf{y}) =\mathbf{y},
\nonumber
\\
\mathbf{Y}_T^{+} ( t,\mathbf{y} ) & \geq&0,\qquad Y
_{T,j}^{+} ( t,\mathbf{y} ) \,dL_{T,j}^{+}
( t,\mathbf{y}%
 ) = 0,\\
  L_{T,j}^{+} ( 0,
\mathbf{y} ) &=&0,\qquad d{L}_{T,j}^{+} ( t,\mathbf{y} )
\geq0.\nonumber
\end{eqnarray}
According to (\ref{SKR_solution}), if $\mathbf{y}=0$,
%
\begin{equation}
\mathbf{Y}_{T}^{+} ( t,0 ) =\max_{0\leq u\leq t}
\bigl(\mathbf{Z}%
_{T}^{\leftarrow} ( t ) -
\mathbf{Z}_{T}^{\leftarrow} ( u ) \bigr). \label{Exp1_Y_pls}
\end{equation}
Since $E[\mathbf{Z}(1)]<0$, the process $\mathbf{Y}^{+}$ satisfying the
Skorokhod problem (\ref{SP_Bnd}) with orthogonal reflection ($R=I$)
possesses a unique stationary distribution. So, we can construct a
stationary version of $ ( \mathbf{Y}^{+} ( s )
\dvtx -\infty<s\leq0 ) $ as
%
\begin{equation}
\mathbf{Y}_*^{+} ( s ) =\lim_{T\rightarrow-\infty}
\mathbf{Y}%
_{T}^{+} ( -T-s,0 ). \label{Ypls_stat}
\end{equation}
The following representation of $\mathbf{Y}_*^{+}(\cdot)$ is known in
the queueing literature; still we include a short proof to make the
presentation self-contained.

\begin{proposition}\label{pr1}
Given any $t\geq0$,
%
\begin{equation}
\mathbf{Y}_*^{+}(-t)=-\mathbf{Z}(t)+\max_{t\leq u<\infty}
\mathbf{Z}(u). \label{Stat_version}
\end{equation}
\end{proposition}

\begin{pf}
Expression (\ref{Exp1_Y_pls}) together with the definition of $\mathbf
{Z}%
_{T}^{\leftarrow} ( \cdot ) $ yields
\begin{eqnarray*}
\mathbf{Y}_{T}^{+} ( -T+s,0 ) & =&\max_{0\leq u\leq-T+s}
\bigl(\mathbf{Z}^{\leftarrow} ( s ) -\mathbf{Z}^{\leftarrow} ( T+u ) \bigr) =
\max_{T\leq r\leq s}\bigl(\mathbf{Z}^{\leftarrow} ( s ) -
\mathbf{Z}^{\leftarrow} ( r ) \bigr)
\\
&=&\max_{T\leq r\leq s}\bigl(-\mathbf{Z} ( -s ) +\mathbf{Z} ( -r )
\bigr) =-\mathbf{Z} ( -s ) +\max_{T\leq r\leq s}\mathbf{Z} ( -r ).
\end{eqnarray*}
Let $-s=t\geq0$ and $-r=u\geq0$, and we obtain
$\mathbf{Y}_{T}^{+} ( -T-t,0 ) =-\mathbf{Z} (
t ) +\max_{t\leq u\leq-T}\mathbf{Z} ( u )$.
Now send $-T\rightarrow\infty$ and arrive at (\ref{Stat_version}), thereby
obtaining the result.
\end{pf}

\subsection{The structure of the main simulation procedure}\label%
{Subsec_Main1}

We now are ready to explain our main algorithm to simulate unbiased samples
from the steady-state distribution of $\mathbf{Y}$. For this purpose,
let us
first define
\[
\mathbf{M} ( t ) =\max_{t\leq u<\infty}\mathbf{Z}(u),
\]
for $t\geq0$ so that $\mathbf{Y}_*^+(-t)=\mathbf{M}(t)-\mathbf{Z}(t)$.
Since $%
E[\mathbf{Z} ( 1 ) ] <0$, it follows that $\mathbf{M}(0)<\infty
$, and hence $ ( \mathbf{M} (
t ) \dvtx t\geq0 ) $ is a stochastic process with finite value. We
assume that we can simulate $\mathbf{M} ( \cdot ) $
jointly with $\mathbf{Z}(\cdot)$ until the coalescence time $\tau$, and
we shall
explain how to perform such simulation procedures in Section~\ref{sec2.3}.

\begin{algorithm}[{[Exact sampling of $\mathbf{Y} ( \infty )$]}]\label{alg1}
\textit{Step} 1: Simulate $(\mathbf{M}(t),\mathbf{Z}(t))$ jointly until time
$\tau\geq0$ such that $\mathbf{Z}(\tau)=\mathbf{M}(\tau)$.

\textit{Step} 2: Set $\mathbf{X}_{-\tau}^{\leftarrow} ( t ) =%
\mathbf{Z}(\tau)-\mathbf{Z} ( \tau-t ) +\mathbf{z}\times t$, and
compute $\mathbf{Y}_{-\tau} ( t,0 ) $ for $0\leq t\leq
\tau$ that solves the Skorokhod problem with input process $\mathbf{X}%
_{-\tau}^{\leftarrow} ( t )$ and initial value $\mathbf{Y}%
_{-\tau}(0,0)=0$. In detail, $\mathbf{Y}_{-\tau} ( t,0%
 ) $ solves
\begin{eqnarray*}
d\mathbf{Y}_{-\tau} ( t,0 ) & =&d\mathbf{X}_{-\tau
}^{\leftarrow}
( t ) +R\,d\mathbf{L}_{-\tau} ( t,0 ),
\\
\mathbf{Y}_{-\tau} ( t,0 ) & \geq& 0,\qquad %
Y_{-\tau,j}
( t,0 ) \,dL_{-\tau,j} ( t,0 ) =0,\\
 L_{-\tau,j} ( 0,0 ) &=&0,\qquad %
dL_{-\tau,j} ( t,0 ) \geq0,
\end{eqnarray*}
for $\tau$ units of time.

\textit{Step} 3: Output $\mathbf{Y}_{-\tau} ( \tau,0 ) $
which has the distribution of $\mathbf{Y} ( \infty ) $.

In step 2, The constant $\mathbf{z}$ is chosen according to Lemma~\ref{LmK_W_Comp} such
that $\mathbf{Z}(t)=\mathbf{X}(t)-\mathbf{z}t$. The time is $-\tau$
precisely the coalescence time as in a DCFTP algorithm. The following
proposition summarizes the validity of this algorithm.
\end{algorithm}

\begin{proposition}\label{pr2}
The previous algorithm terminates with probability one, and its output
is an
unbiased sample from the distribution of $\mathbf{Y} ( \infty
) $.
\end{proposition}

\begin{pf}The argument is similar to the classic Lyones construction.
Let us start by first noting that
\[
\mathbf{Y}_+^*(0)=\mathbf{M}(0)=0\vee\bigl(-U_1\bolds{\mu}+
\mathbf {W}(1)+\mathbf{M}'\bigr). %
\]
Here $U_1$ is the arrival time of the first job and follows an
exponential distribution. $\mathbf{M}'=\max_{0\leq t<\infty} \mathbf
{Z}(t+U_1)-\mathbf{Z}(U_1)<\infty$ is equal in distribution to $\mathbf
{M}(0)$. Then $P(\mathbf{Y}_+^*(0)=0)=P(U_1\geq\max_i (W_i(1)+M'_i)/\mu
_i)>0$ since $U_1$ has infinite support and is independent of both
$\mathbf{W}(1)$ and $\mathbf{M}'$. Therefore, $\mathbf{Y%
}^{+} ( \infty ) $ has an atom at zero. 
This implies that $
\tau<\infty$ with probability one. Actually, we will show later that
$E[\exp(\delta\tau)]<\infty$ for some $\delta>0$ in Theorem~\ref{th1}. Let
$T<0$, and note that, thanks to
Lemma~\ref{LmK_W_Comp}, for $t\in(0,\llvert  T\rrvert ]$
%
\begin{equation}
R^{-1}\mathbf{Y}_{T}(t,0)\leq R^{-1}
\mathbf{Y}_{T}^{+}(t,0).\label{Comparison}
\end{equation}
In addition, by monotonicity of the solution to the Skorokhod problem
in terms
of its initial condition [see \citet{KellaWhitt1996}], we also have
[using the definition of $\mathbf{Y}^+_T(t,\mathbf{y})$ from (\ref
{Ypls_def}) and $\mathbf{Y}^{+}_*(T)$ from (\ref{Ypls_stat})] that
%
\begin{equation}
\mathbf{Y}_{T}^{+}(t,0)\leq\mathbf{Y}_{T}^{+}
\bigl(t,\mathbf{Y}%
^{+}_*(T)\bigr)=\mathbf{Y}_*^{+}(T+t).\label{Comparison_2}
\end{equation}
So $\mathbf{Y}_*^{+}(T+t)=0$ implies $\mathbf{Y}_{T}^{+}(t,0)=0$. One
step further, as $R^{-1}$ has nonnegative coordinates, equations (\ref
{Comparison}) and (\ref{Comparison_2}) imply that $\mathbf{Y}_{T}(t,0)=0$.
Consequently, if $-T>\tau\geq0$,
\[
\mathbf{Y}_{T}\bigl(\llvert T\rrvert -\tau,0\bigr)=0,
\]
which in particular yields that $\mathbf{Y}_{T}(-T,0)=\mathbf{Y}%
_{-\tau}(\tau,0)$. We then obtain that
\[
\lim_{T\to-\infty}\mathbf{Y}_{T}(-T,0)=\mathbf{Y}%
_{-\tau}(
\tau,0),
\]
thereby concluding that $\mathbf{Y}_{\tau}(-\tau,0)$ follows the distribution
$\mathbf{Y} ( \infty ) $ as claimed.
\end{pf}

Step 2 in Algorithm \ref{alg1.1} is straightforward to implement because the
process $%
\mathbf{X}_{-\tau}^{\leftarrow} ( \cdot ) $ is piecewise linear,
and the solution to the Skorokhod problem, namely $\mathbf{Y}_{-\tau
} (
\cdot,0 )$, is also piecewise linear. The gradients are simply
obtained by solving a sequence of linear system of equations which are
dictated by evolving the ordinary differential equations given in (\ref
{S1b}%
). Therefore, the most interesting part is the simulation of the stochastic
object $ ( \mathbf{M} ( t ) \dvtx 0\leq t\leq\tau ) $ in
step 1,
as we will discuss in Section~\ref{sec2.3}.

\subsection{Simulation of the stationary dominating process}\label{sec2.3}

As customary, we use the notation $E_{0} ( \cdot ) $ or $%
P_{0} ( \cdot ) $ to indicate the conditioning $\mathbf{Z} (
0 ) =0$. We define $\phi_i(\theta)=E_0[\exp(\theta Z_i(1))]$ to be the
moment-generating function of $Z_i(1)$, and let $\psi_i(\theta)=\log
(\phi_i(%
\theta))$. In order to simplify the explanation of the simulation procedure
to sample $ ( \mathbf{M} ( t ) \dvtx t\geq0 ) $, we introduce
the following assumption:

\emph{Assumption}: (A3b) Suppose that in every dimension $i$ there
exists $\theta_{i}^{\ast}\in ( 0,\infty ) $ such that
\[
\psi_{i} \bigl( \theta^*_{i} \bigr) =\log E_{0}
\exp \bigl( \theta _{i}^*Z_{i} ( 1 ) \bigr) = 0.
\]

This assumption is a strengthening of assumption (A3), and it is known as
Cramer's condition in the large deviations literature. As we shall explain
at the end of Section~\ref{sec2.3}, it is possible to dispense this assumption and
only work under assumption (A3). For the moment, we continue under assumption
(A3b).

We wish to simulate $ ( \mathbf{Z} ( t ) \dvtx 0\leq t\leq
\tau ) $ where $\tau$ is a time such that
\[
\mathbf{Z}(\tau)=\mathbf{M}(\tau)=\max_{s\geq\tau}\mathbf{Z}(s)\quad
\mbox{and hence}\quad
\forall0\leq t\leq\tau, \qquad\mathbf{M}(t)=\max_{t\leq s\leq\tau
}
\mathbf{Z%
}(s).
\]
Recall that $-\tau$ is precisely the coalescence time since $\mathbf{Y}%
^+_*(-\tau)=0$. We also keep in mind that our formulation at the
beginning of the \hyperref[sec1]{Introduction} implies that
\[
\mathbf{Z}(t)=\mathbf{J} ( t ) -R\mathbf{r}t-\mathbf{z}%
t=\sum
_{k=1}^{N ( t ) }\mathbf{W} ( k ) -R\mathbf{r}t-
\mathbf{z}t,
\]
where $\mathbf{z}$ is selected according to Lemma~\ref{LmK_W_Comp}. Define
\[
\bolds{\mu}=R\mathbf{r}+\mathbf{z},
\]
and let $\mu_{i}>0$ be the $i$th coordinate of $\bolds{\mu}$. In addition,
we assume that we can choose a constant $m>0$ large enough such that
%
\begin{equation}
\sum_{i=1}^{d}\exp \bigl( -
\theta^*_{i}m \bigr) <1. \label{L1}
\end{equation}

Define
%
\begin{equation}
T_{m}=\inf\bigl\{t\geq0\dvtx Z_{i}(t)\geq m,\mbox{for
some } i\bigr\}. \label{T_m}
\end{equation}
Now we are ready to propose the following procedure to simulate $\tau$:

\renewcommand{\thealgorithm}{1.1}
\begin{algorithm}[(Simulating the coalescence time)]\label{alg1.1}
The output of this algorithm is $ ( \mathbf{Z} ( t ) \dvtx 0\leq
t\leq\tau ) $, and the coalescence time $\tau$. Choose the
constance $m$
according to (\ref{L1}):
\begin{longlist}[(1)]
\item[(1)] Set $\tau=0$, $\mathbf{Z} ( 0 ) =0$.

\item[(2)] Generate an inter-arrival time $U$ distributed Exp$(\lambda
)$, and
sample $\mathbf{W}= ( W_{1},\ldots,W_{d} ) $ independent of $U$.

\item[(3)] Let $\mathbf{Z} ( \tau+t ) =\mathbf{Z} ( \tau
 ) -t%
\bolds{\mu}$ for $0\leq t<U$ and $\mathbf{Z} ( \tau+U )
=\mathbf{Z}%
 ( \tau ) +\mathbf{W}-U\bolds{\mu}$.

\item[(4)] If there exists an index $i$, such that $W_{i}-U\mu_{i}\geq-m$,
then return to step~2 and reset $\tau\longleftarrow\tau+U$. Otherwise,
sample a Bernoulli $I $ with parameter $p=P_{0}(T_{m}<\infty)$.

\item[(5)] If $I=1$, simulate a new \textit{conditional path} $ (
\mathbf{C}%
 ( t ) \dvtx 0\leq t\leq T_{m} ) $ following the conditional
distribution of $ ( \mathbf{Z} ( t ) \dvtx 0\leq t\leq
T_{m} ) $
given that $T_{m}<\infty$ and $\mathbf{Z} ( 0 ) =0$. Let $%
\mathbf{Z} ( \tau+t ) =\mathbf{Z} ( \tau ) +\mathbf{C}
 ( t ) $ for $0\leq t\leq T_{m}$, and reset $\tau\longleftarrow
\tau+$
$T_{m}$. Return to step~2.

\item[(6)] Else, if $I=0$, stop and return $\tau$ along with the feed-in
path $%
 ( \mathbf{Z}(t)\dvtx 0\leq t\leq\tau ) $.
\end{longlist}
\end{algorithm}

We shall now explain how to execute the key steps in the previous algorithm,
namely, steps 4 and 5.

\subsubsection{Simulating a path conditional on reaching a positive
level in
finite time}\label{sec2.3.1}

The procedure that we shall explain now is an extension of the
one-dimensional procedure given in \citet{BlanchetSigman2011}; see also the
related one-dimensional procedure by \citet{EnsorGlynn2000}. The strategy
is to use acceptance/rejection. The proposed distribution is based on
importance sampling by means of exponential tilting. In order to describe
our strategy, we need to introduce some notation.

We think of the probability measure $P_{0} ( \cdot ) $ as defined
on the canonical space of right-continuous with left-limits $\mathbb
{R}^{d}$-valued functions, namely, the ambient space of ($\mathbf{Z} (
t )
\dvtx t\geq0)$ which we denote by $\Omega=D_{[0,\infty)} ( \mathbb{R}%
^{d} ) $. We endow the probability space with the Borel $\sigma$-field
generated by the Skorokhod $J_{1}$ topology; see \citet{Billingsley1999}.
Our goal is to simulate from the conditional law of $(\mathbf
{Z}(t)\dvtx 0\leq
t\leq T_m)$ given that $T_m<\infty$ and $\mathbf{Z}(0)=0$, which we
shall denote by $P^*_0$ in the rest of this part.

Now let us introduce our proposed distribution, $P_{0}^{\prime} (
\cdot ) $, defined on the space $\Omega^{\prime}=D_{[0,\infty
)} (
\mathbb{R}^{d} ) \times\{1,2,\ldots,d\}$. We endow the probability space
with the product $\sigma$-field induced by the Borel $\sigma$-field
generated by the Skorokhod $J_{1}$ topology and all the subsets of $%
\{1,2,\ldots,d\}$. So, a typical element $\omega^{\prime}$ sampled
under $%
P_{0}^{\prime} ( \cdot ) $ is of the form $\omega^{\prime
}=((%
\mathbf{Z} ( t ) \dvtx t\geq0),\operatorname{Index})$, where $\operatorname{Index}\in\{
1,2,\ldots,d\}$.
The distribution of $\omega^{\prime}$ induced by $P_{0}^{\prime} (
\cdot ) $ is described as follows. First, set
%
\begin{equation}
P_{0}^{\prime} ( \operatorname{Index}=i ) =w_{i}:=
\frac{\exp (
-\theta_{i}^{\ast}m ) }{\sum_{j=1}^{d}\exp (
-\theta_{j}^{\ast}m ) }. \label{DisK}
\end{equation}
Now, given $\operatorname{Index}=i$, for every set $A\in\sigma(\mathbf{Z} (
s )
\dvtx 0\leq s\leq t)$,
\[
P_{0}^{\prime} ( A|\operatorname{Index}=i ) =E_{0}\bigl[ \exp
\bigl( \theta_{i}^{\ast}Z_{i} ( t ) \bigr)
I_{A}\bigr].
\]
So, in particular, the Radon--Nikodym derivative (i.e., the likelihood ratio)
between the distribution of $\omega=(\mathbf{Z} ( s ) \dvtx 0\leq
s\leq
t)$ under $P_{0}^{\prime} ( \cdot ) $ and $P_{0} ( \cdot
 )
$ is given by
\[
\frac{dP_{0}^{\prime}}{dP_{0}} ( \omega ) =\sum_{i=1}^{d}w_{i}
\exp \bigl( \theta_{i}^{\ast}Z_{i} ( t ) \bigr) .
\]

\textit{The distribution of $(\mathbf{Z} ( s ) \dvtx s\geq0)$ under
$%
P_{0}^{\prime} ( \cdot ) $ is precisely the proposed distribution
that we shall use to apply acceptance/rejection.} It is straightforward to
simulate under $P_{0}^{\prime} ( \cdot ) $. First, sample $\operatorname{Index}$
according to the distribution (\ref{DisK}). Then, conditional on $\operatorname{Index}=i$,
the process $\mathbf{Z} ( \cdot ) $ also follows a compound Poisson
process. Given $\operatorname{Index}=i$, under $P_{0}^{\prime} ( \cdot ) $,
it follows that $\mathbf{J} ( t ) $ can be represented as
%
\begin{equation}
\mathbf{J} ( t ) =\sum_{k=1}^{\hat{N} ( t ) }\mathbf
{W}%
^{\prime} ( k ), \label{J_prime}
\end{equation}
where $\hat{N} ( \cdot ) $ is a Poisson process with rate $%
\lambda E[\exp(\theta^*_iW_i)]$. In addition, the distribution of
$\mathbf{W}^{\prime
}$ is obtained by exponential titling such that for all $A\in\sigma
(\mathbf{W%
})$,
%
\begin{equation}
P^{\prime}\bigl(\mathbf{W}^{\prime}\in A\bigr)=E\bigl[\exp\bigl(
\theta_i^*W_i\bigr)I_A\bigr]. \label{J_P2}
\end{equation}
In sum, conditional on $\operatorname{Index}=i$, we simply let
%
\begin{equation}
\mathbf{Z} ( t ) =\sum_{k=1}^{\hat{N} ( t ) }\mathbf
{W}%
^{\prime} ( k ) -\bolds{\mu}t. \label{J_P3}
\end{equation}

Now, note that we can write
\begin{eqnarray*}
E_{0}^{\prime} \bigl( Z_{\operatorname{Index}} ( t ) \bigr) & =&\sum
_{i=1}^{d}E_{0}
\bigl(Z_{i}(t)\exp \bigl( \theta_{i}^{\ast}Z_{i}
( t ) \bigr) \bigr)P^{\prime} ( \operatorname{Index}=i )
\\
& =&\sum_{i=1}^{d}\frac{d\phi_{i} ( \theta_{i}^{\ast} )
}{d\theta}%
w_{i}>0,
\end{eqnarray*}
where the last inequality follows by convexity of $\psi_{k} (
\cdot ) $ and by definition of $\theta_{k}^{\ast}$. So, we have
that $%
Z_{\operatorname{Index}} ( t ) \nearrow\infty$ as $t\nearrow\infty$ with
probability one under $P_{0}^{\prime} ( \cdot ) $ by the law of
large numbers. Consequently $T_{m}<\infty$ a.s. under $P_{0}^{\prime
} (
\cdot ) $.

Recall that $P_{0}^{\ast} ( \cdot ) $ is the conditional law
of $%
 ( \mathbf{Z} ( t ) \dvtx 0\leq t\leq T_{m} ) $ given that
$%
T_{m}<\infty$ and $\mathbf{Z} ( 0 ) =0$. In order to
assure that we can indeed apply acceptance/rejection theory to simulate
from $P^*_0(\cdot)$, we need to show that the likelihood ratio $%
dP_{0}/dP_{0}^{\prime}$ is bounded:
%
\begin{eqnarray}
\label{ARB} && \frac{dP_{0}^{\ast}}{dP_{0}^{\prime}} \bigl( \mathbf{Z} ( t ) \dvtx 0\leq t\leq
T_{m} \bigr)
\nonumber
\\
&&\qquad =\frac{1}{P_{0} ( T_{m}<\infty ) }\times\frac{dP_{0}}{%
dP_{0}^{\prime}} \bigl( \mathbf{Z} ( t ) \dvtx 0\leq
t\leq T_{m} \bigr)
\\
&&\qquad =\frac{1}{P_{0} ( T_{m}<\infty ) }\times\frac{1}{%
\sum_{i=1}^{d}w_{i}\exp ( \theta_{i}^{\ast}Z_{i} ( T_{m} )
 ) }.\nonumber
\end{eqnarray}
Upon $T_m$, there is an index $L$ ($L$ may be different from $\operatorname{Index}$) such
that $\exp ( \theta_{L}^{\ast}Z_{L} ( T_{m} )  ) \geq
\exp ( \theta_{L}^{\ast}m ) $, therefore
%
\begin{equation}
\frac{1}{\sum_{i=1}^{d}w_{i}\exp ( \theta_{i}^{\ast}Z_{i} (
T_{m} )  ) } \leq\frac{1 }{w_{L}\exp (
\theta_{L}^{\ast}m )}=\sum_{i=1}^{d}
\exp \bigl( -\theta _{i}^{\ast}m \bigr) < 1, \label{B_1}
\end{equation}
where the last inequality follows by (\ref{L1}). Consequently, plugging
(\ref%
{B_1}) into (\ref{ARB}) we obtain that
%
\begin{equation}
\frac{dP_{0}^{\ast}}{dP_{0}^{\prime}} \bigl( \mathbf{Z} ( t ) \dvtx 0\leq t\leq T_{m}
\bigr) \leq\frac{1}{P_{0} ( T_{m}<\infty ) }. \label{ARB_1}
\end{equation}
We now are ready to summarize our acceptance/rejection procedure and the
proof of its validity.

\renewcommand{\thealgorithm}{1.1.1}
\begin{algorithm}[(Simulation of paths conditional on
$T_{m}<\infty$)]\label{alg1.1.1}

\textit{Step} 1: Sample $ ( \mathbf{Z} ( t ) \dvtx 0\leq t\leq
T_{m} ) $ according to $P_{0}^{\prime} ( \cdot ) $ as
indicated via equations~(\ref{DisK}), (\ref{J_prime}) and (\ref{J_P3}).

\textit{Step} 2: Given $ ( \mathbf{Z} ( t ) \dvtx 0\leq t\leq
T_{m} ) $, simulate a Bernoulli $I$ with probability
\[
\frac{1}{\sum_{i=1}^{d}w_{i}\exp ( \theta_{i}^{\ast}Z_{i} (
T_{m} )  ) }.
\]
[Note that the previous quantity is less than unity due to (\ref{B_1}%
).]

\textit{Step} 3: If $I=1$, output $ ( \mathbf{Z} ( t )
\dvtx 0\leq
t\leq T_{m} ) $ and Stop, otherwise go to step 1.
\end{algorithm}

\begin{proposition}\label{pr3}
The probability that $I=1$ at any given call of step 3 in Algorithm
\ref{alg1.1.1} is
$P_0(T_m<\infty)$. Moreover, the output of Algorithm \ref{alg1.1.1} follows the
distribution $P^*_0$.
\end{proposition}

\begin{pf}
The result follows directly from the theory of acceptance/rejection; see
\citet{AsmussenGlynn2007}, pages 39--42. According to it, since the two
probability measures $P_0^*$ and $P'_0$ satisfy
\[
\frac{dP^*_0}{dP'_0}\leq c = \frac{1}{P_0(T_m<\infty)}, %
\]
as indicated by (\ref{ARB}) and (\ref{ARB_1}), one can sample exactly
from $P_0^*$ by the so-called acceptance/rejection procedure:
\begin{longlist}[(1)]
\item[(1)] Generate i.i.d. samples $\{\omega_i\}$ from $P'_0$ and
i.i.d. random numbers $U_i\sim U[0,1]$ independent of $\{\omega_i\}$.
\item[(2)] Define $N=\inf\{n\geq1\dvtx U_n\leq c^{-1}\frac{dP^*_0}{
dP'_0}(\omega_i)\}$.
\item[(3)] Output $\omega_N$.
\end{longlist}
The output $w_N$ follows exactly the law $P^*_0$, and $N$ is a
geometric random variable with mean $c$; in other words, the
probability of accepting a proposal is $c$. In our specific case, we
have $c=1/P_0(T_m<\infty)$, and according to (\ref{ARB}) the likelihood
ration divided by constant $c$ is
\[
c^{-1}\frac{dP^*_0}{ dP'_0}(\omega)=\frac{1}{\sum_{i=1}^{d}w_{i}\exp
 ( \theta_{i}^{\ast}Z_{i} (
T_{m} )  )}. %
\]
Therefore, Algorithm \ref{alg1.1.1} has acceptance probability
$P(I=1)=P_0(T_m<\infty)$, and it generates a path exactly from $P^*_0$
upon acceptance.
\end{pf}

As the previous result shows, the output of the previous procedure
follows exactly the distribution of $ ( \mathbf{Z} ( t )
\dvtx 0\leq t\leq
T_{m} ) $ given that $T_{m}<\infty$ and $\mathbf{Z} ( 0 )
=%
0$. Moreover, the Bernoulli random variable $I$ has probability $%
P_{0} ( T_{m}<\infty ) $ of success. So this procedure actually
allows both steps 4 and 5 in Algorithm \ref{alg1.1} to be executed
simultaneously. In
detail, one simulates a path following the law of $P^{\prime}_0$ until
$T_m$, and then, if the proposed path is accepted, it can be concluded that
$T_m$ is
finite and the proposed path is exactly a sample path following the law
of $%
P^*_0$; otherwise one can conclude that $T=\infty$.

\begin{remark*}
As mentioned earlier, assumption (A3b) is a
strengthening of
assumption (A3). We can carry out our ideas under assumption (A3) as follows.
First, instead of $(\mathbf{M} ( t ) \dvtx t\geq0)$, we consider
the following
process $\mathbf{Z}_{\mathbf{a}}(\cdot)$ and $\mathbf{M}_{\mathbf
{a}} (
\cdot ) $ defined by
\[
\mathbf{Z}_{\mathbf{a}} ( t ):=\mathbf{Z} ( t ) +\mathbf{a%
}t,\qquad
\mathbf{M}_{\mathbf{a}} ( t ) =\max_{s\geq t} \bigl(
\mathbf{Z}%
_{\mathbf{a}} ( s ) \bigr).
\]
We shall explain how to choose the nonnegative vector $%
\mathbf{a}= ( a_{1},a_{2},\ldots,a_{d} ) ^{T}$ in a moment.
Note that we can simulate $ ( \mathbf{M} ( t ) \dvtx t\geq
0 ) $
jointly with $ ( \mathbf{Z} ( t ) \dvtx t\geq0 ) $ if we are
able to simulate $ ( \mathbf{M}_{\mathbf{a}} ( t )
\dvtx t\geq0 ) $ jointly with $(\mathbf{Z}_{\mathbf{a}} ( t )
\dvtx t\geq0)$. Now note that $\psi_{i} ( \cdot ) $ is strictly convex
and that $\dot{\psi}_{i} ( 0 ) <0$, so there exists $a_{i}>0$
large enough to force the existence of $\theta_{i}^{\ast}>0$ such that $
E\exp ( \theta_{i}^{\ast}Z_{i} ( 1 )
+a_{i}\theta_{i}^{\ast} ) =1$, but at the same time small enough to
keep $E ( Z_{i} ( 1 ) +a_{i} ) <0$; again, this
follows by
strict convexity of $\psi_{i} ( \cdot ) $ at the origin. So, if
assumption (A3b) does not hold, but assumption (A3) holds, one can then
execute Algorithm \ref{alg1.1} based on the process $\mathbf{Z_a}(\cdot)$.
\end{remark*}

\subsection{Computational complexity}\label{sec2.4}

In this section we provide a complexity analysis of our algorithm. We first
make some direct observations assuming the dimension of the network remains
fixed. In particular, we note that the expected number of random variables
simulated has a finite moment-generating function in a neighborhood of the
origin.

\begin{theorem}\label{th1}
Suppose that \textup{(A1)} to \textup{(A3)} are in force. Let $\tau$ be the coalescence
time, and
$N$ be the number of random variables generated to terminate the overall
procedure to sample $\mathbf{Y} ( \infty ) $. Then there
exists $%
\delta>0$ such that
\[
E\exp ( \delta\tau+\delta N ) <\infty.
\]
\end{theorem}

\begin{pf}
This follows directly from classical results about random walks; see
\citet{Gut2009}. In particular it follows that $E_{0}^{\prime}(\exp (
\delta T_{m} ) )<\infty$. The rest of the proof follows from elementary
properties of compound geometric random variables arising from the
acceptance/rejection procedure.
\end{pf}

We are more interested, however, in complexity properties as the network
increases. We shall impose some regularity conditions that allow us to
consider a sequence of systems indexed by the number of dimensions $d$. We
shall grow the size of the network in a meaningful way; in particular, we
need to make sure that the network remains stable as the dimension $d$
increases. Additional regularity will also be imposed.

\emph{Assumptions}:

There exists two constants $0<\delta<1<H<\infty$ independent of $d$
satisfying the following conditions:

(C1) $R^{-1}E[\mathbf{X}(1)]<-2\delta R^{-1}\mathbf{1}$ in each network.

(C2) Let $\theta^*_i$ for $i=1,\ldots,d$ be the tilting parameters as
defined in
assumption~(A3b), then
\[
E\exp \bigl[ \bigl(\delta+\theta_i^*\bigr) W_{i} \bigr]
\leq H <\infty
\]
and
\[
H> \delta+\theta^{*}_{i}\qquad\mbox{for all }1\leq i\leq d.
\]

(C3) The arrival rate $\lambda\in(\delta, H)$.

\begin{remark*}
Assumption (C1) implies that $\bolds{\mu}=R\mathbf
{r}+\mathbf{%
z}>\delta\mathbf{1}$, where $\mathbf{z}$ is defined according to Lemma~\ref{LmK_W_Comp}. In
detail, we choose $\mathbf{z}=E[\mathbf{X}(1)]+\delta\mathbf{1}$ and
therefore, $R\mathbf{r}+\mathbf{z}=E[\mathbf{J}(1)]+\delta\mathbf
{1}>\delta%
\mathbf{1}$.

Note that $x\leq\exp(ax)/(ae)$ for any $a>0$ and $x\geq0$. Plugging in
$a=\theta^*_i+\delta$, we have $%
E[W_i]\leq E[\exp((\theta^*_i+\delta)W_i)]/(e(\delta+\theta
^*_i))<H/(e\delta)$ and therefore
\[
\bolds{\mu}=\lambda E[\mathbf{W}]+\delta\mathbf{1}<\bigl(H^2/(e
\delta )+\delta\bigr)\mathbf{1}=H'\mathbf{1},
\]
where $H'=H^2/(e\delta)+\delta$.
Similarly, we also have that $E[W_i^2]\leq E[4\exp((\theta^*_i+\delta
)W_i)]/(e^2(\theta^*_i+\delta)^2)\leq4H/(e^2\delta^2)$, and then we
can compute
\begin{eqnarray*}
E\bigl[Z_i(1)^2\bigr]&=&E\Biggl[\Biggl(\sum
_{k=1}^{N(1)}W_i(k)-\mu_i
\Biggr)^2\Biggr]\leq2E\Biggl[\mu_i^2+\Biggl(
\sum_{k=1}^{N(1)}W_i(k)
\Biggr)^2\Biggr]
\\
&\leq&2\mu_i^2+2\bigl(\lambda+\lambda^2\bigr)
\frac{4H}{e^2\delta^2}\leq 2{H'}^2+\frac{8(H^2+H^3)}{e^2\delta^2}:=H''.
\end{eqnarray*}
In sum, we can conclude that
\[
\max_{1\leq i\leq d}E_0\bigl[Z_{i}(1)^2
\bigr]\leq H''.
\]
In the complexity analysis, we shall only use the fact that $H$, $H'$
and $H''$ are constants independent of $d$. As a result, for the
simplicity of notation, we shall write $H$ for $H$, $H'$ and $H''$
in the rest of this section and assume, without loss of generality, that
\[
\bolds{\mu}\leq H\mathbf{1}\quad\mbox{and}\quad\max_{1\leq i\leq
d}E_0
\bigl[Z_{i}(1)^2\bigr]\leq H. %
\]
\end{remark*}

As discussed in Section~\ref{sec2.3.1}, in Algorithm \ref{alg1.1}, we actually do steps~4 and~5 simultaneously. Therefore, we can rewrite Algorithm \ref{alg1.1} as follows:

\renewcommand{\thealgorithm}{1.1$'$}
\begin{algorithm}[(Simulate the coalescence time)]\label{alg1.1'}
%
\begin{longlist}[(1)]
\item[(1)] Set $\tau=0$, $\mathbf{Z}(0)=0$, $N=0$.

\item[(2)] Simulate a sample from $\mathbf{W}-U\bolds{\mu}$. Here
$U$ is
exponentially distributed with mean $1/\lambda$ and independent of
$\mathbf{W%
}$. Record the value of $\mathbf{Z}(t)$ for $\tau\leq t\leq\tau+U$.
Reset $%
N\leftarrow N+1$, $\mathbf{Z}(\tau+U)\leftarrow\mathbf{Z}(\tau)+\mathbf
{W}-U%
\bolds{\mu}$, $\tau\leftarrow\tau+U$.

\item[(3)] If there exists some index $i$, such that $W_{i}-Ur_{i}\geq-m$,
return to step 2.

\item[(4)] Otherwise, simulate a random walk $\{\mathbf{C}(n)\}$ such
that $%
\mathbf{C}(0)=0$ and $\mathbf{C}(n)=\mathbf{C}(n-1)+\mathbf{W}^{\prime
}(n)-U^{\prime}(n)\bolds{\mu}$, where $\mathbf{W}^{\prime
}(n)-U^{\prime
}(n) \bolds{\mu}$ are independent and identically distributed as
$\mathbf{W}%
^{\prime}-U^{\prime}\bolds{\mu}$ under the tilted measure
$P^{\prime}$
defined in Section~\ref{sec2.3.1} through (\ref{J_prime}) to (\ref{J_P3}). Perform
the simulation until $N_{m}=\inf\{n\geq0\dvtx C_{i}(n)>m\mbox{ for some }i\}$.

\item[(5)] Reset $N\leftarrow N+N_{m}$. Compute $p=1/\sum_{k=1}^{d}
w_{k}\exp
(\theta^{*}_{k} C_{k}(N_{m}))$, and sample a Bernoulli $I$ with
probability $%
p $. If $I=1$, $\mathbf{Z}(\tau+\sum_{k=1}^{N_{m}}U^{\prime
}(k))=\mathbf{Z}%
(\tau)+\mathbf{C}(N_{m})$ and $\tau=\tau+\sum_{k=1}^{N_{m}}U^{\prime}(k)$.
Return to step 2.

\item[(6)] If $I=0$, stop and output $\tau$ with $ ( \mathbf
{Z}(t)\dvtx 0\leq
t\leq\tau ) $.
\end{longlist}
\end{algorithm}

In this algorithm, the total number of random variables required to generate
is $d\cdot N$. Use $N(d)$ instead of $N$ to emphasize the dependence on the
number of dimensions $d$. The following result shows that our algorithm has
polynomial complexity with respect to $d$:

\begin{theorem}\label{th2}
Under assumptions \textup{(C1)} to \textup{(C3)},
\[
E\bigl[N(d)\bigr]=O\bigl(d^{\gamma}\bigr)\qquad\mbox{as }d\to\infty,
\]
for some $\gamma$ depending on $\delta$ and $H$.
\end{theorem}

Denote the number of Bernoulli's generated in step 5 by $N_{b}$ and the
number of random variables generated before executing step 4 in a single
iteration by $N_{a}$. By Wald's identity, we can conclude
\[
E\bigl[N(d)\bigr]=E[N_{b}]\bigl(E[N_{a}]+E[N_{m}]
\bigr).
\]

The following proposition gives an estimate for $E[N_{m}]$.

\begin{proposition}\label{pr4}
Under assumptions \textup{(C1)} to \textup{(C3)},
\[
E[N_{m}]=O(\log d),
\]
and the coefficient in the bound depends only on $\delta$ and $H$.
\end{proposition}

\begin{pf}
First, let us consider the cases in which $W_{i}$ are uniformly bounded from
above by some constant $B$.

Recall that $\phi_i(\theta)=E_0[\exp(\theta Z_i(1))]$. Given $\operatorname{Index}=i$,
one can check that $E_0'[C_i(1)]= \dot{\phi}_i(\theta_i^*)/(\lambda
E[\exp(\theta^*_iW_i)])\geq\dot{\phi}_i(\theta_i^*)/(\lambda H)$.
$N_{m}$ is a stopping time and $C_{i}(N_{m})<m+B$. By the optional
sampling theorem, we have
\[
E[N_{m}]=\sum_{i=1}^d
\omega_i\frac{ E_0'[C_{i}(N_m)]}{E_0'[C_i(1)]}\leq \sum_{i=1}^d
\omega_i\frac{\lambda H(m+B)}{\dot{\phi}_i(\theta_i^*)}.
\]
For each $1\leq i\leq d$, we are going to estimate a lower bound for
$\dot{\phi}(\theta_i^*)$. Using Taylor's expansion around 0, we have
\[
\phi_{i}\bigl(\theta_{i}^{*}\bigr)=
\phi_{i}(0)+\theta_{i}^{*}\dot{
\phi}_{i}%
(0)+\frac{(\theta^{*}_{i})^{2}}{2}\ddot{\phi}_{i}
\bigl(u_1\theta^{*}%
_{i}\bigr),
\]
for some $u_1\in[0,1]$. As $\phi_{i}(\theta_{i}^{*})=\phi_{i}(0)=1$, we have
\[
\theta_{i}^{*}\dot{\phi}_{i}(0)+
\frac{(\theta^{*}_{i})^{2}}{2}%
\ddot{\phi}_{i}\bigl(u_1
\theta^{*}_{i}\bigr)=0.
\]
As $\theta^{*}_{i}>0$,
%
\begin{equation}
\dot{\phi}_{i}(0)+\frac{\theta^{*}_{i}}{2}\ddot{\phi} _{i}
\bigl(u_1\theta^{*}_{i}\bigr)=0.\label{theta}%
\end{equation}
Under assumption (C1), $\dot{\phi}_{i}(0)=E_0[Z_{i}(1)]<-\delta$. Under
assumption (C2), we have that
\begin{eqnarray*}
E_0\bigl[\exp\bigl(\bigl(\delta+\theta_i^*
\bigr)Z_{i}(1)\bigr)\bigr]&\leq&\exp\bigl(\lambda\log\bigl(E\bigl[\exp
\bigl(\bigl(\delta+\theta_i^*\bigr) W_i\bigr)\bigr]\bigr)
\bigr)\\
&\leq& H^\lambda\leq H^H\triangleq H_1<\infty.
\end{eqnarray*}
As a result,
\begin{eqnarray*}
\ddot{\phi}_{i}\bigl(u_1\theta^{*}_{i}
\bigr) & =&E\bigl[Z_{i}(1)^{2}\exp\bigl(u_1\theta
^{*}_{i}Z_{i}(1)\bigr)\bigr]
\\
& \leq& E\bigl[Z_{i}(1)^{2}I\bigl(Z_{i}(1)\leq0
\bigr)\bigr]+E\bigl[Z_{i}(1)^{2}\exp\bigl(
\theta^{*}_{i} 
Z_{i}(1)\bigr)I
\bigl(Z_{i}(0)>0\bigr)\bigr]
\\
& \leq& E\bigl[Z_{i}(1)^{2}\bigr]+E\bigl[Z_{i}(1)^{2}
\exp\bigl(\theta^{*}_{i}Z_{i}(1)\bigr)I
\bigl(Z_{i} 
(0)>0\bigr)\bigr]
\\
&\leq& E\bigl[Z_i(1)^2\bigr]+ E\bigl[Z_{i}(1)^{2}%
\exp\bigl(-\delta Z_{i}(1)\bigr)\cdot\exp\bigl(\bigl(\delta+
\theta_i^*\bigr) Z_{i}(1)\bigr)\bigr].
\end{eqnarray*}
Besides, one can check that for any $x>0$, $x^{2}\exp(-\delta x)\leq
4e^{-2}/\delta^2$. Therefore,
\begin{eqnarray*}
\ddot{\phi}_{i}\bigl(u\theta^{*}_{i}\bigr) &
\leq
& E\bigl[Z_{i}(1)^{2}\bigr]+\frac
{4}{\delta^2}e^{-2}E
\bigl[\exp\bigl(\bigl(\delta+\theta_i^*\bigr) Z_{i}(1)
\bigr)\bigr]
\\
& \leq& H+\frac{4}{\delta^2}e^{-2}H_1.
\end{eqnarray*}
Plug this result into equation (\ref{theta}) and use that $\dot{\phi
}_i(0)<-\delta$ to complete the inequality
%
\begin{equation}
\theta^{*}_{i}\geq\frac{2\delta}{H+4e^{-2}H_1/\delta^2}.\label{thetastar}
\end{equation}

On the other hand, by a Taylor expansion of $\phi_i(\cdot)$ around
$\theta_i^*$, we can conclude that
%
\begin{equation}
\dot{\phi}_i\bigl(\theta^*_i\bigr)=\frac{\theta^*_i}{2}
\ddot{\phi}\bigl(u_2\theta _i^*\bigr),\label{phi_prime}
\end{equation}
for some $u_2\in[0,1]$.
Note that
\begin{eqnarray*}
\ddot{\phi}_{i}\bigl(u_2\theta^{*}_{i}
\bigr) & =&E_0\bigl[Z_{i}(1)^{2}\exp
\bigl(u_2\theta ^{*}_{i}Z_{i}(1)
\bigr)\bigr] \geq E_0\bigl[Z_{i}(1)^{2}\exp
\bigl(u_2\theta ^*_iZ_i(1)\bigr)I(U>1)\bigr]
\\
&\geq& E\bigl[\mu_i^2\exp\bigl(-\theta^*_i
\mu_i\bigr)I(U>1)\bigr]\geq\mu_i^2\exp(-H\mu
_i)\exp(-\lambda)\\
&\geq&\delta^2\exp\bigl(-H^2-H
\bigr).
\end{eqnarray*}
Thus (\ref{thetastar}) together with (\ref{phi_prime}) imply
%
\begin{equation}
\dot{\phi}_{i}\bigl(\theta^{*}_{i}\bigr)\geq
\frac{1}{2}\theta^{*}_{i} \delta^{2}
e^{-H^2-H}%
\geq\frac{\delta^3 e^{-H^2-H}}{H+4e^{-2}H_1/\delta^2}.\label{phi_prime_bound}
\end{equation}
Note that for lower bound (\ref{phi_prime_bound}) to hold, we do not
require $W_i$ to be bounded.

Therefore,
\[
E[N_{m}]\leq\sum_{i=1}^d
\omega_i\frac{
\lambda H(m+B)}{\dot{\phi}_i(\theta_i^*)} \leq\frac{\lambda H(m+B)(H+4e^{-2}H_1/\delta^2)}{\delta^3 e^{-H^2-H}%
}, %
\]
as $\omega_i>0$ and $\sum_i \omega_i=1$.

By (\ref{thetastar}), we have that $\theta^*_i$ are all uniformly
bounded away from 0, so we can choose $m=O(\log d/\min_{i}
\theta^{*}_{i})=O(\log d)$ to satisfy equation (\ref{L1}). Now we can
conclude that $E[N_{m}]=O(\log d)$ as $B$, $H$ and $\delta$ are all
constants independent of~$d$.

Now, let us consider the more general cases when the $W_{i}$'s are not
bounded from
above. Recall that $\mathbf{W}'$ is derived from $\mathbf{W}$ by
exponential tilting; see (\ref{J_P2}). For any $B>0$, define $\tilde
{\mathbf{W}}'$ by $\tilde{W}'_{i}=W'_{i}I(W'_i\leq B)$ as the
truncation of $\mathbf{W}'$, and define the random walk $\tilde
{C}_{i}(n)=\tilde{C}%
_{i}(n-1)+\tilde{W}'_{i}(n)-U'(n)\mu_{i}$. Let $\tilde{N}_{m}=\inf\{n\dvtx \tilde{C}_{i}(n)>m\mbox{ for some }i\}$. Since $\tilde{C}_{i}(n)\leq
C_{i}%
(n)$, we have $\tilde{N}_{m}\leq N_{m}$. Our goal is to show that one
can choose a proper value for $B$ such that $E[\tilde{N}_m]=O(\log d)$
and hence so is $E[N_m]$.

Since $\tilde{W}'_{i}$ is
bounded from above by $B$, by the optimal stopping theorem, we have
\[
E[\tilde{N}_{m}]\leq\sum_{i=1}^d
\omega_i\frac{m+B}{E[\tilde{C}_i(1)]}.
\]
By definition,
\[
E\bigl[\tilde{C}_i(1)\bigr]= E\bigl[\bigl(W_iI(W_i
\leq B)-U\mu_i\bigr)\exp\bigl(\theta _i^*
\bigl(W_iI(W_i\leq B)-U\mu_i\bigr)\bigr)
\bigr]. %
\]
Since $U\mu_{i}\geq0$, we have
\begin{eqnarray*}
&& E\bigl[\bigl(W_{i}I(W_i\leq B)-U\mu_{i}
\bigr)\exp\bigl(\theta^{*}_{i} \bigl(W_{i}I(W_i
\leq B)-U\mu_{i}\bigr)\bigr)\bigr]
\\
&&\qquad\geq  E\bigl[(W_{i}-U\mu_i)\exp\bigl(\theta^{*}_{i}(W_{i}-U
\mu_{i})\bigr)\bigr]-E\bigl[W_{i}\exp \bigl(\theta^{*}_{i}W_{i}
\bigr)I(W_{i}>B)\bigr].
\end{eqnarray*}
By assumption (C2), $\delta$ and $H>0$ are constants independent of $d$
such that
\[
E \bigl[\exp\bigl(\bigl(\delta+\theta^*_i\bigr) W_i
\bigr)\bigr]\leq H<\infty. %
\]
As a consequence,
\begin{eqnarray*}
E\bigl[W_{i}\exp\bigl(\theta^{*}_{i}W_{i}
\bigr)I(W_{i}>B)\bigr]&\leq& E\bigl[W_i\exp(-\delta
W_i)I(W_i>B)\exp\bigl(\bigl(\delta+\theta_i^*
\bigr) W_i\bigr)\bigr]
\\
&\leq&\max_{w>B}\bigl\{w\exp(-\delta w)\bigr\}E\bigl[\exp
\bigl(\bigl(\delta+\theta_i^*\bigr) W_i\bigr)\bigr]\\
&\leq& B
\exp(-\delta B)H
\end{eqnarray*}
for all $B>1/\delta$.
Recall that by (\ref{phi_prime_bound}),
\begin{eqnarray*}
E\bigl[(W_{i}-U\mu_i)\exp\bigl(\theta^{*}_{i}(W_{i}-U
\mu_i)\bigr)\bigr]&=&E\bigl[C_i(1)\bigr]\geq\dot {
\phi}_i\bigl(\theta_i^*\bigr)/(\lambda H)\\
&\geq&
\frac{\delta^3 e^{-H^2-H}}{\lambda
H(H+4e^{-2}H_1/\delta^2)},
\end{eqnarray*}
where $H_1=H^H$. Therefore, we can take $B=O(-\frac{1}{\delta}\log(\frac
{\delta^3 e^{-H^2-H}}{2\lambda H^2(H+4{\delta}e^{-2}H_1/\delta^2)}))$
independent
of $d$ such that
\begin{eqnarray*}
B\exp(-\delta B)H&<&\frac{\delta^3 e^{-H^2-H}}{2\lambda
H(H+4e^{-2}H_1/\delta^2)}\quad \mbox{and hence}\\
 E\bigl[
\tilde{C}_i(1)\bigr]&\geq& \frac{\delta^3 e^{-H^2-H}}{2\lambda H(H+4e^{-2}H_1/\delta^2)}. %
\end{eqnarray*}
In the end, since $m=O(\log(d))$, we have
\[
E[N_{m}]\leq E[\tilde{N}_{m}]\leq\frac{ 2\lambda
H(m+B)(2H+8e^{-2}H_1/\delta^2)}{\delta^3 e^{-H^2-H}}=O(\log
d).
\]
\upqed\end{pf}
Now we give the proof of the main result in this subsection.
\begin{pf*}{Proof of Theorem~\ref{th2}}
Recall that
\[
E[N]=E[N_b]\bigl(E[N_{a}]+E[N_{m}]\bigr).
\]
Since $N_b$ is the number of trials required to obtain $I=0$, $E[N_b]=
1/P(I=0)$. As discussed in Section~\ref{sec2.3.1}, $P(I=0)\geq1-\sum_{i=1}^{d}
\exp(-\theta^{*}_{i} m)$ and hence
\[
E[N_b]\leq\frac{1}{1-\sum_{i=1}^{d} \exp(-\theta^{*}_{i} m)}\leq\frac
{1}{1-{1}/{d}}
\]
if we take $m=2\log d/\min_{i}\theta^{*}_{i}$.

Similarly, we have $E[N_{a}]=1/P(U>(m+W_{i})/\mu_{i}, \forall  i)$. For
any $K>0$,
\[
P\biggl(U>\frac{m+W_{i}}{\mu_{i}}, \forall i\biggr)\geq P\biggl(U>\frac{m+K}{\min_{i} \mu_{i}};
W_{i}\leq K\mbox{ for all }i\biggr).
\]
Under assumption (C2), we have
\[
P(W_{i}\leq K\mbox{ for all }i)\geq1-\sum
_{i=1}^{d}P(W_{i}>K)\geq 1-dH\exp{(- K
\delta)}.
\]
Under assumption (C3), we have
\[
P\biggl(U>\frac{m+K}{\min_i\mu_i}\biggr)\geq\exp\biggl(-\frac{H(m+K)}{\min_i\mu_i}\biggr).
\]
As $U$ and $\mathbf{W}$ are independent,
\[
P\biggl(U>\frac{m+W_{i}}{\mu_{i}}, \forall i\biggr)\geq\exp \biggl( -
\frac
{H(m+K)}{\min_{i}
\mu_{i}} \biggr) \bigl(1-dH\exp(-K\delta)\bigr).
\]
Choosing
$K=(2\log{d}+\log{H})/\delta$ and plugging in $m=2\log d/\min_i\theta
_i^*$, we get
\[
E[N_{a}]\leq\frac{1}{1-{1}/{d}} d^{(2H/(\min_{i} \mu_{i}\min
_{i}\theta^{*}_{i})+2H/(\delta\min_i\mu_i))} H^{H/(\delta
\min_i\mu_i)}.
\]

By Proposition~\ref{pr4} we
have $E[N_{m}]=O(\log d)$. In summary, we have
\begin{eqnarray*}
E[N]&=&E[N_b]\bigl(E[N_{a}]+E[N_{m}]\bigr)=O
\biggl(\biggl(\frac{1}{1-{1}/{d}}\biggr)^2 \log d d^{{2H}/{(\min_{i}
\mu_{i}\min_{i}\theta^{*}_{i})}}
\biggr)
\\
&=&O\bigl(d^{1+{2H}/{(\min_{i} \mu_{i}\min
_{i}\theta^{*}_{i})}}\bigr).
\end{eqnarray*}

As discussed in the proof of Proposition~\ref{pr4}, $\theta^{*}_{i}\geq
\delta/(H+4e^{-2}H_1/\delta^2)$ and $\mu_{i}\geq\delta$ are
uniformly bounded away from $0$, therefore,
\[
E[N]=O\bigl(d^{1+{2H(H+4e^{-2}H_1/\delta)}/{\delta^2}}\bigr).
\]
\upqed\end{pf*}

\section{Extension to Markov-modulated processes}
\label{SectionExtension}

We shall briefly explain how our development in Section~\ref{SecCompPoi}, specifically
Algorithm \ref{alg1}, can be implemented beyond input with stationary and
independent increments. As an example, we shall concentrate on
Markov-modulated stochastic fluid networks. Our extension to Markov-modulated
networks is first explained in the one-dimensional case, and later we will
indicate how to treat the multidimensional setting.

Let $ ( \hat{I} ( t ) \dvtx t\geq0 ) $ be an irreducible
continuous-time Markov chain taking values on the set $\{1,\ldots,n\}$. We
assume that, conditional on $\hat{I} ( \cdot ) $, the number of
arrivals, $\hat{N} ( \cdot ) $, follows a time-inhomogeneous
Poisson process with rate $\lambda_{\hat{I} ( \cdot ) }$. We
further assume that $\int_{0}^{t}\lambda_{\hat{I} ( s ) }\,ds>0$
with positive probability. The process $\hat{N} ( \cdot ) $ is
said to be a Markov-modulated Poisson process with intensity $\lambda
_{\hat{I} ( \cdot ) }$. Define $\hat{A}_k$ to be the time of
the $k$th arrival, for $k\geq1$; that is, $\hat{A}_k=\inf\{t\geq0\dvtx \hat{N
} ( t ) =k\}$.

We assume that the $k$th arrival brings a job requirement equal to $%
\hat{W} ( k ) $. We also assume that the $\hat{W} (
k ) $'s are conditionally independent given the process $\hat
{I} ( \cdot ) $. Moreover, we assume that the moment-generating
function $\phi_{i} ( \cdot ) $ defined via
\[
\phi_{i} ( \theta ) =E \bigl( \exp \bigl( \theta\hat{X} ( k ) \bigr)
|\hat{I} ( \hat{A}_k ) =i \bigr),
\]
is finite in a neighborhood of the origin. In simple words, the job
requirement of the $k$th arrival might depend upon the environment, $%
\hat{I} ( \cdot ) $, at the time of arrival. But, conditional on
the environment, the job sizes are independent. Finally, we assume that the
service rate at time $t$ is equal to $\mu_{\hat{I} ( t ) }\geq0$.

Let $\hat{X}(t)=\sum_{k=1}^{\hat{N}(t)}\hat{W}({k})-\int_{0}^{t}\mu
_{\hat{I} (s)}\,ds$. Then the workload process, $(Y(t)\dvtx t\geq0)$, can
be expressed as
\[
Y(t)=\hat{X}(t)-\inf_{0\leq s\leq t}\hat{X}(s),
\]
assuming that $Y(0)=0$. In order for the process $Y (
\cdot ) $ to be stable, in the sense of having a stationary
distribution, we assume that $\sum_i\pi_i(\lambda_iE[
\hat{W}|\hat{I}=i]-\mu_i)<0$, where $\pi_i$ is the stationary
distribution of the Markov chain $\hat{I}$. Following the same argument
as in Section~\ref{SecCompPoi}, we can
construct a stationary version of the process $Y ( \cdot ) $
by a
time reversal argument.

Since $\hat{I}(\cdot)$ is irreducible, one can define its associated
\textit{stationary} time-reversed Markov chain ${I(\cdot)}$ with transition
rate matrix $\mathcal{A}$; for the existence and detailed description of
such reversed chain, see Chapter~2.5 of \citet{Asmussen2003}. Let us
write $%
N ( \cdot ) $ to denote a Markov-modulated Poisson process with
intensity $\lambda_{I ( \cdot ) }$, and let $A_{k}=\inf
\{t\geq0\dvtx N ( t ) =k\}$. We consider a sequence $ ( W (
k ) \dvtx k\geq1 ) $ of conditionally independent random variables
representing the service requirements (backward in time) such that $%
\phi_{i} ( \theta ) =E ( \exp ( \theta W ( k )
 ) |I ( A_{k} ) =i ) $.

We then can define $Z(t)=\sum_{k=1}^{N(t)}W ( k )
-\int_{0}^{t}\mu_{I(s)}\,ds$. Following the same arguments as in
Section~\ref{SecCompPoi}, we
can run a stationary version $Y^{\ast}$ of $Y$ backward via the process
\[
Y^{\ast}(-t)=\sup_{s\geq t} \bigl( Z(s)-Z(t) \bigr).
\]
Therefore, $Y^{\ast}(-t)$ can be simulated exactly as long as a convenient
change of measure can be constructed for the process $ ( I (
\cdot ),Z(\cdot) ) $, so that a suitable adaptation of Algorithm
\ref{alg1.1.1} can be applied. Once the adaptation of Algorithm \ref{alg1.1.1} is in place,
the adaptation of Algorithms~\ref{alg1.1} and~\ref{alg1} is straightforward.

In order to define such change of measure, let us define the matrix $%
\mathcal{M}(\theta,t)\in\mathbb{R}^{n\times n}$, for $t\geq0$, via
\[
\mathcal{M}_{ij}(\theta,t)=E_{i}\bigl[\exp\bigl(\theta
Z(t)\bigr);I(t)=j\bigr],
\]
where the notation $E_{i} ( \cdot ) $ means that $I (
0 )
=i$. Note that $\mathcal{M}(\cdot,t)$ is well defined in a neighborhood of
the origin. In what follows we assume that $\theta$ is such that all
coordinates of $\mathcal{M}%
(\theta,t)$ are finite.

It is known [see, e.g., Chapters~11.2 and~13.8 of %
\citet{Asmussen2003} and the references therein] that $\mathcal{M}%
(\theta,t)=\exp(tG(\theta))$ where the matrix $G$ is
defined by
\[
G_{ij}(\theta)=%
\cases{\mathcal{A}_{ij}, &\quad $\mbox{if }i\neq j$, \vspace*{2pt}
\cr
\mathcal{A}_{ii}-
\mu_{i}\theta+\lambda_{i}\phi_{i}(\theta), &\quad $\mbox{if }i=j.$}
\]
Besides, $G(\theta)$ has a unique eigenvalue $\beta(\theta)$
corresponding to a strictly positive eigenvector $(u ( i,\theta
 )
\dvtx 1\leq i\leq n)$. The eigenvalue $\beta(\theta)$ has the following
properties which follow from Propositions~2.4 and~2.10 in
Chapter~11.2 of \citet{Asmussen2003}:

\begin{lemma}
\label{PropSummaryMM}
\begin{longlist}[(1)]
\item[(1)] $\beta(\theta)$ is convex in $\theta$ and $\dot{\beta}(\theta)$ is well defined.

\item[(2)] $\lim_{t\rightarrow\infty}Z(t)/t=\dot{\beta}(0)=\lim_{t\rightarrow
\infty}\hat{X} ( t ) /t<0$.

\item[(3)] $ ( M ( t,\theta ) \dvtx t\geq0 ) $ defined via
\[
M ( t,\theta ) =\frac{u(I ( t ),\theta)}{u(I (
0 ),\theta)}\exp\bigl(\theta Z(t)-t\beta(\theta)\bigr)
\]
is a martingale.
\end{longlist}
\end{lemma}

As explained in Chapter~13.8 of \citet{Asmussen2003}, the martingale $%
M ( \cdot ) $ induces a change of measure for the process
$ (
I ( \cdot ),Z(\cdot) ) $ as we shall explain. Let $P$ be the
probability law of $(I(\cdot),Z(\cdot))$, and define a new probability
measure $\tilde{P}$ for $(I(s),Z(s)\dvtx s\leq t)$ as $d\tilde{P}=M (
t,\theta ) \,dP$.

We now describe the law of $ ( I ( \cdot ),Z (
\cdot )  ) $ under $\tilde{P}$. The process $I(\cdot)$ is a
continuous time Markov chain with rate matrix $\widetilde{\mathcal
{A}}%
_{ij}=\mathcal{A}_{ij}u(j,\theta)/u ( i,\theta ) $ for $i\neq j$
(and $\widetilde{\mathcal{A}}_{ii}=-\sum_{j\neq i}\widetilde{\mathcal
{A}}%
_{ij}$). In addition,
\[
Z(t)\stackrel{d} {=}\sum_{k=1}^{\tilde{N}(t)}
\tilde{W} ( k ) -\int_{0}^{t}\mu_{I ( s ) }\,ds,
\]
where $\tilde{N}$ is a Markov-modulated Poisson process with rate at
time $t
$ equal to $\phi_{I ( t ) } ( \theta ) \lambda
(I(t))$, and
the $\tilde{W} ( k ) $'s are conditionally independent given $%
I ( \cdot ) $ with moment generating function $\widetilde{\phi}
_{i} ( \cdot ) $ defined via
\[
\widetilde{\phi}_{i} ( \eta;\theta ) =\widetilde{E}\bigl(\exp\bigl(
\eta \tilde{W} ( k ) \bigr)|A_{k}=i\bigr)=\phi_{i} ( \eta+
\theta ) /\phi_{i} ( \eta ),
\]
which is finite in a neighborhood of the origin. In addition, $%
Z(t)/t\rightarrow\dot{\beta}(\theta)$ under~$\tilde{P}$.

Because of the stability condition of the system, we have that $\dot
{\beta}(0)<0$. Then, following the same argument as in the remark given
at the end
of Section~\ref{sec2.3}, we may assume the existence of the Cramer root $\theta
^{\ast
}>0$ such that $\beta(\theta^{\ast})=0$ and $\dot{\beta} (
\theta^{\ast} ) >0$. The change of measure that allows adaption of
Algorithm \ref{alg1.1.1} is given by selecting $\theta^{\ast}>0$ as indicated. Now,
select $m>0$ such that
%
\begin{equation}
K:=\exp \bigl( -\theta^{\ast}m \bigr) \max_{i,j}
\frac{u ( i,\theta
^{\ast} ) }{u ( j,\theta^{\ast} ) }\leq1. \label{SelectionK}
\end{equation}
We will use the notation $P_{0,i} ( \cdot ) $ to denote the
law $%
P ( \cdot ) $ conditional on $Z ( 0 ) =0$ and $I (
0 ) =i$. Let us write $P_{0,i}^{\ast} ( \cdot ) $ to denote
the law of $(Z ( t ) \dvtx 0\leq t\leq T_{m})$ [under $P_{0,i} (
\cdot ) $] conditional on $T_{m}<\infty$. Further, we write $\tilde
{P}%
_{0,i} ( \cdot ) $ to denote the law of $\tilde{P} (
\cdot ) $, selecting $\theta=\theta^{\ast}$, conditional on $Z (
0 ) =0$ and $I ( 0 ) =i$. Then we have that $\tilde{P}%
_{0,i} ( T_{m}<\infty ) =1 $ [by Lemma~\ref{PropSummaryMM}
since $%
\dot{\beta} ( \theta^{\ast} ) >0$], and therefore [by (\ref%
{SelectionK})], we have
\begin{eqnarray*}
&&\frac{dP_{0,i}^{\ast}}{d\tilde{P}_{0,i}} \bigl( \bigl( I ( t ) ,Z ( t ) \bigr) \dvtx 0\leq t\leq
T_{m} \bigr) \\
&&\qquad=\frac{u (
i,\theta^{\ast} ) }{u ( I ( T_{m} ),\theta^{\ast
} ) }%
\times\frac{\exp ( -\theta^{\ast}Z ( T_{m} )  ) I (
T_{m}<\infty ) }{P_{0,i} ( T_{m}<\infty ) }
\\
&&\qquad\leq\frac{K}{P_{0,i} ( T_{m}<\infty ) }\leq\frac
{1}{P_{0,i} (
T_{m}<\infty ) }.
\end{eqnarray*}
It is clear from this identity, which is completely analogous to
identities (%
\ref{ARB}) and~(\ref{ARB_1}), which are the basis for Algorithm \ref{alg1.1.1}, that
the corresponding adaptation to our current setting follows.

For the $d$-dimensional case ($d>1$), we first assume the existence of the
Cramer root $\theta_{j}^{\ast}>0$ for each dimension $j\in\{1,\ldots,d\}
$. In
this setting we also must compute the corresponding positive
eigenvector $(
u_{j}( i,\theta_{j}^{\ast}) \dvtx 1\leq i\leq n) $ for each $j\in\{1,\ldots
,d\}$.
The desired change of measure that allows the adaptation of Algorithm \ref{alg1.1.1}
is just a mixture of changes of measures such as those described above
induced by $M( \cdot,\theta_{j}^{\ast}) $ in each direction, just as
discussed in Section~\ref{sec2.3.1}, with weight $w_{j}=\exp( -\theta_{j}^{\ast}m)
/\sum_{k=1}^{m}\exp( -\theta_{k}^{\ast}m) $. The corresponding likelihood
ratio is then
\begin{eqnarray*}
&&\frac{dP_{0,i}^{\ast}}{d\tilde{P}_{0,i}} \bigl( \bigl( I ( t ) ,Z ( t ) \bigr) \dvtx 0\leq t\leq
T_{m} \bigr)
\\
&&\qquad=\frac{1}{\sum_{j=1}^{d}w_{j}\exp ( \theta_{j}^{\ast}Z_{j} (
T_{m} )  ) u_{j} ( I ( T_{m} )
,\theta_{j}^{\ast} ) /u_{j} ( i,\theta_{j}^{\ast} ) },
\end{eqnarray*}
and $m$ must be selected so that
\[
\sum_{j=1}^{d}\exp \bigl( -
\theta_{j}^{\ast}m \bigr) \sup_{j,i,k}
\frac{%
u_{j} ( i,\theta_{j}^{\ast} ) }{u_{j} ( k,\theta_{j}^{\ast
} ) }\leq1.
\]

\section{Algorithm for reflected Brownian motion}
\label{SecRBM}

In this section, we revise our algorithm and explain how we can apply
it to
the case of reflected Brownian motion. Consider a multidimensional Brownian
motion
\[
\mathbf{X} ( t ) =\mathbf{v}t+A\mathbf{B} ( t ),
\]
where $\mathbf{v}\in\mathbb{R}^{d}$ is the drift vector, and $A\cdot
A^{T}\triangleq\Sigma\in\mathbb{R}^{d\times d}$ is the positive definite
covariance matrix. Our target process $\mathbf{Y}(t)$ is the solution
to the
following Skorokhod problem with input process $\mathbf{X}(\cdot)$ and
initial value $\mathbf{Y}(0)=\mathbf{y}_0$:
\begin{eqnarray*}
d\mathbf{Y} ( t ) & =&d\mathbf{X} ( t ) +R\,d\mathbf {L} ( t ), \qquad\mathbf{Y} (
0 ) =\mathbf{y}_{0},
\\
\mathbf{Y} ( t ) & \geq&0,\qquad Y_{j} ( t ) \,d L_{j} ( t )
\geq0,\qquad L_{j} ( 0 ) =0,\qquad %
dL_{j} ( t
) \geq0.
\end{eqnarray*}
We assume that the reflection matrix $R$ is an $M$-matrix of the form $%
R=I-Q^T$, where $Q$ has nonnegative coordinates and a spectral radius equal
to $\alpha<1$ so that $R^{-1}$ has only nonnegative elements; see page 304
of \citet{HarrisonReiman1981}. We also assume the stability condition $%
R^{-1}\mathbf{v<0}$ for the existence of the steady-state distribution. As
discussed in the \citet{HarrisonReiman1981}, there is a unique solution
pair $(\mathbf{Y},\mathbf{L})$ to the Skorokhod problem associated with
$%
\mathbf{X}$, and the process $\mathbf{Y}$ is called a reflected Brownian
Motion (RBM). We wish to sample $\mathbf{Y} ( \infty ) $ (at least
approximately, with a pre-defined controlled error).

The stochastic dominance result for reflected Brownian motions that is
analogous to Lemma~\ref{LmK_W_Comp} was first developed in the proof of
Lemma 12 in \citet{HarrisonWilliams1987}. In detail, we can construct
a dominating process $\mathbf{Y}%
^+(\cdot)$ as follows. First, we can choose $\mathbf{z}\in\mathbb{R}^d $
such that $\mathbf{v<z}$ and $R^{-1}\mathbf{z<0}$. Define a process
%
\begin{equation}
\mathbf{Z} ( t ) =\mathbf{X} ( t ) -\mathbf {z}t:=A\mathbf{B%
} ( t )-
\bolds{\mu}t, \label{BMinput}
\end{equation}
where $\bolds{\mu}=\mathbf{v-z}$, and let $\mathbf{Y}^+(\cdot)$ be
the RBM
corresponding to the Skorokhod problem (\ref{SP_Bnd}), which has orthogonal
reflection. Then $R^{-1}\mathbf{Y} ( t ) \leq R^{-1}\mathbf{Y}%
^{+} ( t ) $. As a result, we can assume without loss of generality
that the input Brownian motion has strictly negative drift
coordinatewise. In sum, the following assumption is in force
throughout this
section:

\renewcommand{\theass}{(D)}
\begin{ass}\label{assD}
The input process $\mathbf{Z}(\cdot)$
satisfies (\ref%
{BMinput}) with $\mu_i>\delta_0>0$ for all $1\leq i\leq d$, and we assume
that $A$ is nondegenerate so that $A^TA$ is positive definite.
\end{ass}

Since $\mathbf{Z}(\cdot)$ has strictly negative drift, following the same
argument given for Proposition~\ref{pr1}, we can construct a stationary version of
the dominating process as
%
\begin{equation}
\mathbf{Y}^+(-t)=-\mathbf{Z}(t)+\max_{u\geq t}\mathbf{Z}(u)
\triangleq \mathbf{Z}(t)-\mathbf{M}(t)\qquad \mbox{for all }t\geq0. \label{domBM}
\end{equation}
In order to apply the same strategy as in Algorithm \ref{alg1} to the RBM,
we need to address two problems. First, the input process $\mathbf{Z}$
requires a continuous path description while the computer can only encode
and generate discrete objects. Second, the dominating process is a reflected
Brownian motion with orthogonal reflection. Therefore the hitting time
$\tau$
to the origin is almost surely infinity [see \citet
{VaradhanWilliams1985}%
], which means that Algorithm \ref{alg1} will not terminate in finite time,
in this case. To solve the first problem, we take advantage of a wavelet
representation of Brownian motion and use it to simulate a piecewise linear
approximation with uniformly small (deterministic) error. To solve the
second problem, we define an approximated coalescent time $%
\tau_{\varepsilon} $ as the first passage time to a small ball around the
origin so that $E[\tau_{\varepsilon}]<\infty$ and the error caused by
replacing $\tau$ with $\tau_{\varepsilon}$ is bounded by $\varepsilon$. In
sum, we concede to an algorithm that is not exact but one that could give
any user-defined $\varepsilon$ precision. Nevertheless, at the end of
Section~\ref{sec4.1} we will show that we can actually use this $\varepsilon$-biased
algorithm to estimate without any bias the steady-state expectation of
continuous functions of RBM by introducing an extra randomization step.

Section~\ref{SecRBM} is organized as follows. In Section~\ref{sec4.1}, we will describe the main
strategy of our algorithm. In Section~\ref{SubWave}, we use a wavelet representation
to simulate a piecewise linear approximation of Brownian motion. In
Section~\ref{sec4.3}, we will discuss the details in simulating jointly $\tau
_{\varepsilon}$
and the stationary dominating process based on the techniques we have
already used for the compound Poisson cases. In the end, in
Section~\ref{SubSecCCRBM}, we
will give an estimate of the computational complexity of our algorithm.

\subsection{The structure of the main simulation procedure}\label{sec4.1}

The main strategy of the algorithm is almost the same as Algorithm \ref{alg1},
except for two modifications due to the two issues discussed above: first,
instead of simulating the input process $\mathbf{Z}$ exactly, we
simulate a
piecewise linear approximation $\mathbf{Z}^{\varepsilon}$ such that $%
|Z^{\varepsilon}_{i}(t)-Z_{i}(t)|<\varepsilon$ for all indices $i$ and $t\geq
0$; second,
instead of sampling the coalescence time $\tau$ such that $\mathbf
{M}(\tau)=%
\mathbf{Z}(\tau)$, we simulate an approximation coalescence time, $%
\tau_{\varepsilon}$, such that $\mathbf{M}(\tau_{\varepsilon})\leq\mathbf{Z}%
(\tau_{\varepsilon})+\bolds{\varepsilon}$.

With this notation, we now give the structure of our algorithm. The
details will be given later in Sections \ref{SubWave} and \ref{sec4.3}:

\renewcommand{\thealgorithm}{2}
\begin{algorithm}[{[Sampling with controlled error of $\mathbf{Y} (
\infty )$]}]\label{alg2}

\textit{Step} 1: Let $\tau_{\varepsilon}\geq0$ be any time for which
$\mathbf{M}%
 ( \tau_{\varepsilon} ) \leq\mathbf{Z} ( \tau_{\varepsilon}
) +%
\bolds{\varepsilon} $, and simulate, jointly with $\tau_{\varepsilon}$,
$\mathbf{Z}%
_{-\tau_{\varepsilon}}^{\leftarrow} ( t ) =-\mathbf{Z}%
^{\varepsilon} ( \tau_{\varepsilon}-t ) $ for $0\leq t\leq\tau
_{\varepsilon}$%
.

\textit{Step} 2: Define $\mathbf{X}_{-\tau_{\varepsilon}}^{\leftarrow} (
t ) =\mathbf{Z}^{\varepsilon}(\tau_{\varepsilon})-\mathbf{Z}^{\varepsilon
} (
\tau_{\varepsilon}-t ) +\mathbf{z}t$, and compute $\mathbf{Y}%
^{\varepsilon}_{-\tau_{\varepsilon}} ( \tau_{\varepsilon},0 ) $
which is obtained by evolving the solution $\mathbf{Y}^{\varepsilon}_{-\tau
_{%
\varepsilon}} ( \cdot,0 ) $ to the Skorokhod problem
\begin{eqnarray*}
d\mathbf{Y}^{\varepsilon}_{-\tau_{\varepsilon}} ( t,0 ) & =&d%
\mathbf{X}_{-\tau_{\varepsilon} }^{\leftarrow} ( t ) +R\,d\mathbf{L} 
_{-\tau}
( t,0 ),
\\
\mathbf{Y}^{\varepsilon}_{-\tau_{\varepsilon}} ( t,0 ) & \geq&0,%
\qquad Y^{\varepsilon}_{-\tau_{\varepsilon},j} ( t,0 ) \,dL_{-\tau_{\varepsilon},j}
( t,0 ) \geq0, \\
L_{-\tau_{\varepsilon},j} ( 0,0 )& =&0,\qquad dL_{-\tau
_{\varepsilon},j} ( t,0 )
\geq0,
\end{eqnarray*}
for $\tau_{\varepsilon}$ units of time.

\textit{Step} 3: Output $\mathbf{Y}^{\varepsilon}_{-\tau_{\varepsilon}}(\tau_{
\varepsilon},0)$.
\end{algorithm}

First, we show that there exists a stationary version $\{\mathbf{Y}%
^{*}(t)\dvtx t\leq0\}$ that is coupled with the dominating stationary
process $\{%
\mathbf{Y}^{+}(t)\dvtx t\leq0\}$ as given by (\ref{domBM}).

\begin{lemma}\label{le3}
There exists a stationary version $\{\mathbf{Y}^*(t)\dvtx t\leq0\}$ of
$\mathbf{Y}
$ such that $R^{-1}\mathbf{Y}^{*}(t)\leq R^{-1}\mathbf{Y}^{+}(t)$ for
all $%
t\leq0$.
\end{lemma}

\begin{pf}
The proof follows the same argument as that of Proposition~\ref{pr2}.
\end{pf}

The following proposition shows that the error of the above algorithm
has a
small and deterministic bound.

\begin{proposition}\label{pr5}
Suppose $\mathbf{X}\in\mathbb{R}^{d}$. Let 
$r=\max_{i,j}R^{-1}_{ij}/\min_{i,j}\{R_{ij}^{-1}\dvtx\break  R^{-1}_{ij}>0\}$. Then there exists a stationary version
$\mathbf{Y}^{*}$of $\mathbf{Y}$ such that in each index $i$,
\[
\bigl|Y_{i}^{*}(0)-Y^{\varepsilon}_{\tau_{\varepsilon},i}(
\tau_{\varepsilon},0)\bigr|\leq\biggl( 
\frac{1}{1-\alpha}+dr\biggr)
\varepsilon.
\]
Here $0\leq \alpha <1$ is the spectral radius of the matrix $Q$.%
\end{proposition}

\begin{pf}
Consider three processes on $[-\tau_{\varepsilon},0]$. The first is the coupled
stationary process $\mathbf{Y}^{*}(\cdot)$ as constructed in Lemma~\ref{le3},
which is
the solution to the Skorokhod problem with initial value $\mathbf{Y}^{*}
(-\tau_{\varepsilon})$ at time $-\tau_{\varepsilon}$ and input process
$\tilde{\mathbf{X}}(\cdot)=\mathbf{X}(\tau_{\varepsilon})-\mathbf{X(-\cdot
)}$ on
$[-\tau_{\varepsilon},0]$; the second is a process $\tilde{\mathbf
{Y}}(\cdot)$,
which is the solution to the Skorokhod problem with initial value $0$
at time
$-\tau_{\varepsilon}$ and input process $\tilde{\mathbf{X}}(\cdot)$; the
third is
the process $\mathbf{Y}^{\varepsilon}_{-\tau_{\varepsilon}}(t,0)$ as we
described in
the algorithm, which is the solution to the Skorokhod problem with initial
value $0$ at time $-\tau_{\varepsilon}$ and input process $\mathbf{X}%
_{-\tau_{\varepsilon} }^{\leftarrow} ( t ) $ as defined in step 2
of Algorithm \ref{alg2}.\vadjust{\goodbreak}

By definition, we know that for each index $i$, $|Y_{i}^{+}
(-\tau_{\varepsilon})|<\varepsilon$. Since\break $R^{-1}\mathbf{Y}(\tau_{\varepsilon
})\leq
R^{-1}\mathbf{Y}^{+}(\tau_{\varepsilon})$, the coupled process $Y_{i}^{*}%
(-\tau_{\varepsilon})<dr\,\varepsilon$. Note that $\mathbf{Y}^{*}(\cdot)$ has the
same input data as $\tilde{\mathbf{Y}}(\cdot)$ except for their initial values.
According to the comparison theorem of \citet{Ramasubramanian2000}, the
difference between these two
processes is uniformly bounded by the difference of their initial
values coordinate-wise. Therefore, we can conclude $|Y^{*}_{i}(0)-\tilde
{Y}%
_{i}(0)|<dr\,\varepsilon$.

On the other hand, $\tilde{\mathbf{Y}}(\cdot)$ and $\mathbf{Y}^{\varepsilon
}_{-\tau_{\varepsilon}}(\cdot,0)$ have common initial value 0 and input processes
whose difference is uniformly bounded by $\varepsilon$. It was proved in
\citet{HarrisonReiman1981} that the Skorokhod mapping is Lipschitz\break 
continuous under the uniform metric $d_T(Y^1(\cdot),Y^2(\cdot
))\triangleq\break \max_{1\leq i\leq d}\sup_{0\leq t\leq
T}|Y_i^1(t)-Y_i^2(t)|$ for all $0<T<\infty$, and the Lipschitz constant
is equal to $1/(1-\alpha)$, where $0\leq\alpha<1$ is the spectral
radius of $Q$. Therefore, we have that $|\tilde{Y}_{i}(0)-Y^{\varepsilon
}_{-\tau_{\varepsilon},i}%
(\tau_{\varepsilon},0)|<\varepsilon/(1-\alpha)$.

Simply applying the triangle inequality, we obtain that
\[
\bigl|Y_{i}^{*}(0)-Y^{\varepsilon}_{\tau_{\varepsilon},i}(
\tau_{\varepsilon},0)\bigr|\leq \biggl(\frac{1}{1-\alpha}+dr\biggr)\varepsilon.
\]
\upqed\end{pf}

We conclude this subsection by explaining how to remove the $\varepsilon$%
-bias induced by Algorithm \ref{alg2}. Let $T$ be any positive random variable with
positive density $\{f(t)\dvtx t\geq0\}$ independent of $\mathbf{Y}^*(0)$.
Let $g\dvtx %
\mathbb{R}^d\to\mathbb{R}$ be any positive Lipschitz continuous function
such that there exists constant $K>0$ and for all $\mathbf{x}$ and
$\mathbf{y%
}\in\mathbb{R}^d$, $|g(\mathbf{x})-g(\mathbf{y})|\leq K\max_{i=1}|x_i-y_i|$%
. As illustrated in \citet{Beskosetal2012},
\begin{eqnarray*}
E \bigl[g\bigl(\mathbf{Y}^*(0)\bigr) \bigr]&=&E \biggl[\int_0^{g(\mathbf
{Y}^*(0))}\,dt
\biggr]%
=E \biggl[\int_0^{g(\mathbf{Y}^*(0))}
\frac{f(t)}{f(t)}\,dt \biggr]
\\
&=&E \biggl[\frac{1(g(\mathbf{Y}^*(0))>T)}{f(T)} \biggr].
\end{eqnarray*}
Since $|Y_{i}^{*}(0)-Y^{\varepsilon}_{\tau_{\varepsilon},i}(\tau_{\varepsilon
},0)|%
\leq(1+dr)\varepsilon$, we can sample $T$ first, and then select $\varepsilon>0$
small enough, output $1(g(\mathbf{Y}^{\varepsilon}_{\tau_{\varepsilon}}(\tau_{
\varepsilon},0))>T)/f(T)$ as an unbiased estimator of $E[g(\mathbf{Y}^*(0))]$
without the need for computing $\mathbf{Y}^*(0)$ exactly. It is
important to
have $ (\mathbf{Y}^{\varepsilon}_{\tau_{\varepsilon}}(\tau_{\varepsilon},0)\dvtx %
\varepsilon>0 ) $ coupled as $\varepsilon\to0$, and this can be achieved
thanks to the wavelet construction that we will discuss next.

\subsection{Wavelet representation of Brownian motion}
\label{SubWave}

In this part, we give an algorithm to generate piecewise linear
approximations to a Brownian motion path-by-path, with uniform
precision on
any finite time interval. The main idea is to use a wavelet representation
for Brownian motion.

By the Cholesky decomposition, any multidimensional Brownian motion can be
expressed as a linear combination of independent one-dimensional Brownian
motions. Our goal is to give a piecewise linear approximation to a $d$%
-dimensional Brownian motion $\mathbf{Z}$ with uniform precision
$\varepsilon$
on $[0,1]$. Suppose that we can write $\mathbf{Z}=A\mathbf{B}$, where
$A$ is
the Cholesky decomposition of the covariance matrix, and the $B_{i}$'s are
independent standard Brownian motions. If we are able to give a piecewise
linear approximation $\tilde{B}_{i}$ to each $B_{i}$ on $[0,1]$ with
precision $\varepsilon/(d\cdot a)$ where $a=\max_{i,j}|A_{ij}|$, then
$A\tilde{%
\mathbf{B}}$ is a piecewise linear approximation to $\mathbf{Z}$ with
uniform error~$\varepsilon$. Therefore, in the rest of this part, we only need
to work with a standard one-dimensional Brownian motion.

Now let us introduce the precise statement of a wavelet representation of
Brownian motion; see \citet{Steele2001}, pages 34--39. First we need to define
step function $H(\cdot)$ on $[0,1]$ by
\[
H(t)=%
\cases{1, &\quad $\mbox{for }0\leq t<\tfrac{1}{2},$ \vspace*{2pt}
\cr
-1, &\quad $\mbox{for }\tfrac{1}{2}\leq t\leq1,$ \vspace*{2pt}
\cr
0, & \quad $\mbox{otherwise.}$}
\]
Then define a family of functions
\[
H_{k}(t)=2^{j/2}H\bigl(2^{j}t-l\bigr)
\]
for $k=2^{j}+l$ where $j>0$ and $0\leq l\leq2^{j}$. Set $H_{0}(t)=1$. The
following wavelet representation theorem can be seen in \citet{Steele2001}:

\begin{theorem}\label{th3}
If $\{W^{k}\dvtx 0\leq k<\infty\}$ is a sequence of independent standard normal
random variables, then the series defined by
\[
B_{t}=\sum_{k=0}^{\infty} \biggl(
W^{k}\int_{0}^{t}H_{k}(s)
\,ds \biggr)
\]
converges uniformly on $[0,1]$ with probability one. Moreover, the
process $%
\{B_{t}\}$ defined by the limit is a standard Brownian motion on $[0,1]$.
\end{theorem}

Choose $\eta_{k}=4\cdot\sqrt{\log k}$, and note that $P(|W^{k}|>\eta
_{k})=O(k^{-4})$, so\break $\sum_{k=0}^{\infty}P(|W^{k}|>\eta_{k})<\infty$.
Therefore, $P(|W^{k}|>\eta_{k},\mathrm{i.o.})=0$. The simulation strategy will
be to
sample $\{W^{k}\}$ jointly with the finite set $\{k\dvtx |W^{k}|\geq\eta
_{k}\}$.

Note that if we take $j=\lceil\log_{2}{k}\rceil$, as shown in %
\citet{Steele2001},
\[
\sum_{k=1}^{\infty} \biggl( W^{k}
\int_{0}^{t}H_{k}(s) \,ds \biggr) \leq
\sum_{j=0}^{\infty} \Bigl(2^{-j/2}\cdot
\max_{2^{j}\leq k\leq2^{j+1}-1} \bigl|W^{k}\bigr| \Bigr).
\]
Since $\sum_{j=0}2^{-j/2}\sqrt{j+1}<\infty$, for any $\varepsilon>0$ there
exists $K_{0}>0$, such that
%
\begin{equation}
\sum_{j=\lceil\log{K_{0}}\rceil}2^{-j/2}\sqrt{j+1}<\varepsilon.
\label{definition of K_0}
\end{equation}
As a result, define
%
\begin{equation}
K=\max\bigl\{k\dvtx \bigl|W^{k}\bigr|>\eta_{k}\bigr\}\vee
K_{0}<\infty, \label{definiation of K}
\end{equation}
then $\sum_{k=K+1}^{\infty}|W^{k}|\int_{0}^{t}H_{k}(s)\,ds\leq\varepsilon$.
If we
can simulate $\{(W^{k})_{k=1}^{K},K\}$ jointly,
%
\begin{equation}
B^{\varepsilon}(t)=\sum_{k=0}^{K}W^{k}
\int_{0}^{t}H_{k}(s)\,ds \label{EQ-BE}
\end{equation}
will be a piecewise linear approximation to a standard Brownian motion
within precision $\varepsilon$ in $C[0,1]$.

Now we show how to simulate $K$ jointly with $\{W^k\dvtx 1\leq k\leq K\}$. The
algorithm is as below with $\rho=4$ as we have chosen $\eta_k=4\cdot
\sqrt{%
\log k}$:

\renewcommand{\thealgorithm}{2\normalfont{w}}
\begin{algorithm}[(Simulate $K$ jointly with $\{W^k\}$)]\label{alg2w}

\textit{Step} 0: Initialize $G=K_0$ and $S$ to be an empty array.

\textit{Step} 1: Set $U=1$, $D=0$. Simulate $V\sim \operatorname{Uniform}(0,1)$.

\textit{Step} 2: While $U>V>D$, set $G\leftarrow G+1$ and $U\leftarrow
P(|W^{G}|\leq
\rho\sqrt{\log{G}})\times U$ and $D\leftarrow$ $(1-G^{1-\rho
^{2}/2})\times U$.

\textit{Step} 3: If $V\geq U$, add $G$ to the end of $S$, that is, $S=[S,G]$,
and return
to step~1.

\textit{Step} 4: If $V\leq D$, $K=\max(S,K_0)$.

\textit{Step} 5: For every $k\in S$, generate $W^{k}$ according to the conditional
distribution of $Z$ given $\{|W|>\rho\sqrt{\log{k}}\}$; for other
$1\leq
k\leq K$, generate $W^{k}$ according to the conditional distribution of $W$
given $\{|W|\leq\rho\sqrt{\log{k}}\}$.
\end{algorithm}

In this algorithm, we keep an array $S$, which is used to record the indices
such that $|W^{k}|>\rho\sqrt{\log k}$, and a number $G$ which is the next
index to be added into $S$. Precisely speaking, given that the last element
in array $S$ is $N$, say, $\max(S)=N$, $G=\inf\{k\geq N+1\dvtx |W^{k}|>\rho
\sqrt{\log k}\}$. The key part of the algorithm is to simulate a Bernoulli
with success parameter $P(G<\infty)$ and to sample $G$ given $G<\infty$.

For this purpose, we keep updating two constants $U$ and $D$ such that $
U>P(G=\infty)>D$ and $(U-D)\to0$ as the number of iterations grows. To
illustrate this point, denote the value of $U$ and $D$ in the $m$th
iteration by $U_m$ and $D_m$, respectively. Then for all $m>0$,
\[
P(G=\infty)=\prod_{k=N+1}^{\infty}P
\bigl(\bigl|W^k\bigr|\leq\rho\sqrt{\log{k}}%
\bigr)<\prod
_{k=N+1}^{N+m}P\bigl(\bigl|W^k\bigr|\leq\rho\sqrt{
\log{k}}\bigr)=U_m.
\]
On the other hand, for all $\rho>\sqrt{2}$ and $N$ large enough,
\begin{eqnarray*}
\prod_{k=N+m+1}^{\infty}P\bigl(\bigl|W^k\bigr|
&\leq&\rho\sqrt{\log{k}}\bigr)>1-\sum_{k=N+m+1}^{%
\infty}P
\bigl(\bigl|W^k\bigr|>\rho\sqrt{\log k}\bigr)\\
&\geq&1-(N+m+1)^{1-\rho^2/2},
\end{eqnarray*}
and hence we conclude that $D_m=(1-(N+m+1)^{1-\rho
^{2}/2})U_m<P(G=\infty)$.
Because $(1-(N+m+1)^{1-\rho^2/2})\to1$ as $m\to\infty$, the algorithm
proceeds to steps 3 or 4 after a finite number of iterations, and
we can
decide whether $G<\infty$ or not.

Now we show that we can actually sample $G$ simultaneously as the\break Bernoulli
with success probability $P(G<\infty)$ is generated. If $V<D$, we conclude
that $V<P(G=\infty)$ and hence $G=\infty$ and $K=\max(S)$. Otherwise, we
have $G<\infty$. In this case, suppose step 2 ends in the $(m+1)$th
iteration and $V>U$. Since $U_m=P(|W^k|\leq\rho\sqrt{\log k}\mbox{ for
}%
k=K+1,\ldots,K+m )$, $U_{m+1}\leq V<U_m$ implies nothing but that $%
K+m+1=\inf\{k\geq K+1\dvtx |W^{k}|>\rho\sqrt{\log k}\}$. Therefore, by
definition, $G=K+m+1$ and should be added into array $S$. Once $S$ and $K$
are generated, $\{W^k\dvtx 1\leq k\leq K\}$ can be generated jointly with
$S$ and
$K$ according to step 5.

Also we note that $B^{\varepsilon}(t)$ has the following nice property:

\begin{proposition}\label{pr6}
\[
B^{\varepsilon}(1)=B(1).
\]
\end{proposition}

\begin{pf}
The equality follows from the fact that $\int_{0}^{1}
H_{n}(s)\,ds=0$ for any $n\geq1$ and $m\geq1$.
\end{pf}

As a consequence of this property, for any compact time interval $[0,T]$
(without loss of generality, assume $T$ is an integer), in order to
give an
approximation for $B(t)$ on $[0,T]$ with guaranteed $\varepsilon$ precision
uniformly in $[0,T]$, we only need to run the above algorithm $T$ times to
get $T$ i.i.d. sample paths $\{B^{\varepsilon,(i)}(t)\dvtx t\in [0,1]\}$ for $%
i=1,2,\ldots,T$, and define recursively
\[
B^{\varepsilon}(t)=\sum_{i=1}^{\lfloor t\rfloor}B^{\varepsilon
,(i)}(1)+B_{\lfloor
t\rfloor}^{\varepsilon}
\bigl(t-\lfloor t\rfloor\bigr).
\]

\subsection{A conceptual framework for the joint simulation of \texorpdfstring{$\tau_{\varepsilon}$}{tauvarepsilon} and 
\texorpdfstring{$\mathbf{Z}^{\varepsilon}$}{Zvarepsilon}}\label{sec4.3}

Our goal now is to develop an algorithm for simulating $\tau_{\varepsilon
}$ and
$ ( \mathbf{Z}^{\varepsilon}(t)\dvtx 0\leq t\leq\tau_{\varepsilon} ) $
jointly. In detail, we want to simulate $\mathbf{Z}^{\varepsilon}(t)$ forward
in time and stop at a random time $\tau_{\varepsilon}$ such that for any
time $%
s>\tau_{\varepsilon}$, $Z_{i}(s)\leq Z_{i}(\tau_{\varepsilon})+\varepsilon$ for $
1\leq i\leq d$.

Because of the special structure of the wavelet representation used in
simulating the process $\mathbf{Z}^\varepsilon(\cdot)$, the time $%
T_m\triangleq\inf\{t\geq0\dvtx Z^{\varepsilon}_i(t)>m\mbox{ for some }1\leq
i\leq
d\}$ is no longer a stopping time with respect to the filtration generated
by $\mathbf{Z}(\cdot)$. As a consequence, we cannot directly carry out
importance sampling as in Algorithm~\ref{alg1.1.1}. To remedy this problem, we
decompose the process $\mathbf{Z}^{\varepsilon}(t)$ into two parts: a
random walk $\{\mathbf{Z}^{\varepsilon}(n)\dvtx n\geq0\}$ with Gaussian increment
and a series of independent Brownian bridges $\{\bar{\mathbf{B}}%
_n(s)\triangleq\mathbf{Z}^{\varepsilon}(n+s)-\mathbf{Z}^{\varepsilon}(n)\dvtx s\in
[0,1]%
,n\geq0\}$. Our strategy is to first carry out the importance sampling
as in
Algorithm \ref{alg1.1.1} to the random walk $\{\mathbf{Z}^{\varepsilon}(n)\dvtx n\geq0\}
$ to
find its upper bound, and next develop a new scheme to control the upper
bounds attained in the intervals $\{(n,n+1)\dvtx n\geq0\}$ for the i.i.d.
Brownian bridges $\{\bar{\mathbf{B}}_n(s)\dvtx s\in[0,1],n\geq0\}$.

The whole procedure is based on the wavelet representation of Brownian
motion. Let $\{W_{n}^{k}(i)\dvtx n,k\in\mathbb{N}, i=1,2,\ldots,d\}$ be a sequence
of i.i.d. standard normal random variables. According to the expression
given in Theorem~\ref{th3}, for any $t=n+s$, $s\in[0,1]$,
%
\begin{eqnarray}\label{wv}
Z_{i}(t)&=&Z_{i}(n)+s\bigl(Z_i(n+1)-Z_i(n)
\bigr)
\nonumber
\\[-8pt]
\\[-8pt]
\nonumber
&&{}+\sum_{j=1}^{d} A_{ij}
\Biggl(\sum_{k=1}^{\infty}W_{n}^{k}(j)
\int_{0}^{s} H_{k}(u)\,du\Biggr).
\end{eqnarray}
Let us put (\ref{wv}) in matrix form,
\[
\mathbf{Z}(t)=\mathbf{Z}(n)+s\bigl(\mathbf{Z}(n+1)-\mathbf{Z}(n)\bigr)+A\sum
_{k=1}^{%
\infty}\mathbf{W}_{n}^{k}
\cdot\int_{0}^{s} H_{k}(u)\,du.
\]
For all $n\geq0$ and $s\in[0,1]$, $\bar{\mathbf{B}}_{n}(s)=
A\sum_{k=1}^{\infty}\mathbf{W}_{n}^{k}\cdot\int_{0}^{s}
H_{k}(u)\,du$. Then the sequence $\{\bar{\mathbf{B}}_{n}(\cdot)\dvtx n\geq0\}$
is i.i.d.
Note that $(Z_i(n+1)-Z_i(n))$ is independent of $\{W_{n}^{k}(i)\dvtx k\geq1\}$.
We can split the simulation into two independent parts:
\begin{longlist}[(1)]
\item[(1)] Simulate the discrete-time random walk $\{\mathbf
{Z}(n)\dvtx n\geq0\}$
with i.i.d. Gaussian increments and $\mathbf{Z}(0)=0$. That is, $Z_{i}(0)=0$
and $Z_i(n+1)=Z_i(n)+\sum_{j=1}^{d} A_{ij}W_{n+1}^{0}(j)-\mu_i$, where
$%
\{W_n^0(j)\dvtx n\geq0\}$ are i.i.d. standard normals.

\item[(2)] For each $n$, simulate $\bar{\mathbf{B}}_{n}(s)$ to do bridging
between $\mathbf{Z}(n)$ and $\mathbf{Z}(n+1)$.
\end{longlist}

Now, any time $t_0>0$ is an approximate coalescence time $\tau_\varepsilon
$ if
there exists some positive constant $\zeta>0$ such that the following two
conditions hold for all $n\geq t_0$: Condition (1), $\mathbf{Z}(n)\leq
\mathbf{Z}(t_0)-\zeta(n-\lceil t_0\rceil)\mathbf{1}+\bolds{\varepsilon
}$, and
condition (2), $\max\{\bar{\mathbf{B}}_{n}(s)\dvtx s\in[0,1]\}\leq\zeta
(n-\lceil
t_0\rceil)\mathbf{1}$. Based on these observations, we develop an algorithm
to simulate the approximate coalescence time $\tau_{\varepsilon}$ jointly
with $%
\{ \mathbf{Z}^{\varepsilon}(t)\dvtx 0\leq t\leq\tau_{\varepsilon}\} $.

By Assumption \ref{assD}, $\mu_{i}>\delta_{0}$ for some $\delta_{0}>0 $. Let $%
\zeta=\delta_{0}/2$, and define $\mathbf{S}(n)=\mathbf{Z}(n)+n\bolds
{\zeta}%
\mathbf{1}$ such that $\{\mathbf{S}(n)\dvtx n\geq0\}$ is a random walk with
strictly negative drift. Therefore, condition (1) can be checked by carrying
out the importance sampling procedure as in Algorithm \ref{alg1.1.1} for the random
walk $\{\mathbf{S}(n)\dvtx n\geq0\}$. More precisely, since $S_i(n)$ has
Gaussian increments, we can compute explicitly that $\theta^*_i=2(\mu_i-
\zeta)/\sigma_i$ and choose $m>0$ satisfying (\ref{L1}) in order to carry
out the importance sampling procedure for the random walk $\{\mathbf{S}%
(n)\dvtx n\geq0\}$. Suppose we use the importance sampling procedure and
find $%
t_0$ such that $\mathbf{S}(n)\leq\mathbf{S}(t_0)$ for all $n\geq t_0$, and
hence condition (1) is satisfied for $t_0$.

About condition (2), recall that $\bar{\mathbf{B}}_{n}(\cdot)$'s are i.i.d.
linear combinations of Brownian bridges, and let $M$ be a random time, finite
almost surely, such that
%
\begin{equation}
M\geq\max\Bigl\{n\geq t_0\dvtx \max_{0\leq s\leq1}\bigl(
\bar{B}_{n,i}(s)-\zeta (n-t_0)\bigr)>0%
\mbox{ for
some }i\Bigr\}. \label{Def_M}
\end{equation}

Observe that for $t_0$ to be an approximate coalescence time, conditions (1)
and~(2) must hold simultaneously. If for time $t_0$, for example,
condition~(1) is satisfied while condition (2) is not, we need to
continue the
testing procedure and simulation of the process for $t>t_0$. Then, however,
the random walk $\{\mathbf{S}(n)\dvtx n\geq\lceil t_0\rceil\}$ should be
conditioned on that $\mathbf{S}(n)\leq\mathbf{S}(t_0)$ for the fact that
condition (1) holds for $t_0$ reveals ``additional information'' on the random
walk for $n\geq t_0$. Therefore, such ``additional information'' or
``conditioning event'' must be incorporated and tracked when conditions
(1) and
(2) are sequentially tested. All of these conditioning events are described
and accounted for in Section~\ref{sec4.3.2}, which also includes the overall
procedure to sample $\tau_{\varepsilon}$ jointly with $\mathbf{Z}^{\varepsilon}$.

Now, let us first provide a precise description of $M$ and explain the
simulation algorithm for $M$ in Section~\ref{sec4.3.1}.

\subsubsection{Simulating $M$ and \texorpdfstring{$\{\bar{\mathbf{B}}^{\protect\varepsilon}_{n}(\cdot)\dvtx 1\leq n\leq M\}$}
{\{bar{{B}}{varepsilon}{n}(cdot):1<=n<=M\}}}\label{sec4.3.1}

Recall that $\bar{\mathbf{B}}_{n}(t)= \break  A\sum_{k=1}^{\infty}\mathbf{W}%
_{n}^{k}\cdot\int_{0}^{t} H_{k}(u)\,du$, where $\{W_{n}^{k}(i)\dvtx n\geq
0,k\geq
1, 1\leq i\leq d\}$ are i.i.d. standard normals. Note that
\[
\sum_{n=0}\sum_{k=1} P
\bigl(\bigl|W_{n}^{k}(i)\bigr|\geq4\sqrt{\log(n+1)}+4\sqrt {\log
k%
} \bigr) \leq\sum_{n=0} \sum
_{k=1}\frac{1}{((n+1)k)^{4}} <\infty.
\]

By the Borel--Cantelli lemma, we can conclude that for each $i\in\{
1,\ldots,d\}$
there exists $M^{i}<\infty$ such that for all $(n+1)k>M^{i} $, $%
|W_{n}^{k}(i)|\leq4\sqrt{\log(n+1)}+4\sqrt{\log k}$. Clearly, $\sqrt
{\log t}%
=o(t)$ as $t\rightarrow\infty$, so we can select a $m_{0}$ large enough such
that for any $n>m_{0}$,
\[
(n+1)\zeta-ad\Biggl(4\sqrt{\log(n+1)}-\sum_{j=1}^{\infty}2^{-j}
\sqrt{j}\Biggr)\geq0.
\]
Note that $M^i$ can be simulated jointly with $(W_n^k(i)\dvtx n\geq0,k\geq1,
1\leq i\leq d, (n+1)k\leq M^i)$ by adapting Algorithm \ref{alg2w} in Section~\ref{SubWave} and $M^i$'s are independent of each other. Then, for any $%
n>\max_{i=1}^{d}M^{i}\vee m_{0}$,
\begin{eqnarray*}
\bar{\mathbf{B}}_{n}(t) & =&A\sum_{k=1}^{\infty}
\mathbf{W}_{n}^{k}\cdot \int_{0}^{t}H_{k}(u)\,du
\\
& \leq& ad\Biggl(4\sqrt{\log(n+1)}+\sum_{j=1}^{\infty}2^{-j/2}
\sqrt{j}\Biggr)\leq (n+1)\zeta,
\end{eqnarray*}
where, $j=\lceil\log_{2}{k}\rceil$. Therefore, we can choose $%
M=\max_{i}M^{i}\vee m_{0}$.

Now we introduce a variation of Algorithm \ref{alg2w} that will be used in the
procedure to simulate $M$ and $\{\bar{B}_n^{\varepsilon}(\cdot)\dvtx 1\leq
n\leq M\}$
jointly. In the following algorithm, a~sequence of ``conditioning
events'' of
the form $|W^k|\leq\beta_k$, for some given constants $\{\beta^k\dvtx \beta
^k>4%
\sqrt{\log k}\}$, is in force. Let $\Phi(a)=P(|W|<a)$ for all $a>0$,
where $W
$ is a standard normal. The random number $K$ to be simulated is
defined as in~(%
\ref{definiation of K}).

\renewcommand{\thealgorithm}{2\normalfont{w}$'$}
\begin{algorithm}[(Simulate $K$ jointly with $\{W^k\dvtx 1\leq k\leq K\}$
conditional on $|W^k|\leq\beta^k$ for all $k\geq1$)]\label{alg2w'}

\textit{Step} 0: Initialize $G=K_0$ as defined in (\ref{definition of K_0}) and $S$
to be an empty array.

\textit{Step} 1: Set $U=1$, $D=0$. Simulate $V\sim \operatorname{Uniform}(0,1)$.

\textit{Step} 2: While $U>V>D$, set $G\leftarrow G+1$ and $U\leftarrow\frac{\Phi
(4%
\sqrt{\log G})}{\Phi(\beta^k)}\times U$ and $D\leftarrow$
$(1-G^{-7})\times
U$.

\textit{Step} 3: If $V\geq U$, add $G$ to the end of $S$, that is, $S=[S,G]$,
and return
to step~1.

\textit{Step} 4: If $V\leq D$, $K=\max(S,K_0)$.

\textit{Step} 5: For every $k\in S$, generate $W^{k}$ according to the conditional
distribution of $Z$ given $\{4\sqrt{\log{k}}<|W|\leq\beta^k\}$; for
other $%
1\leq k\leq K$, generate $W^{k}$ according to the conditional distribution
of $W$ given $\{|W|\leq4\sqrt{\log{k}}\}$.
\end{algorithm}

The main difference between Algorithm \ref{alg2w'} and the original Algorithm
\ref{alg2w} is
that $U$ and $V$ are now computed from the conditional probability; however,
the relations $U>V>D$ and $U-D\to0$ still hold, and hence Algorithm \ref{alg2w'}
is valid. Based on this, we can now give the main procedure to simulate $M$
and $\{\bar{B}_n^{\varepsilon}(\cdot)\dvtx 1\leq n\leq M\}$ jointly:

\renewcommand{\thealgorithm}{2\normalfont{m}}
\begin{algorithm}[(Simulating of $M$ and $\{\bar{\mathbf{B}}%
_{n}^{\varepsilon}(\cdot)\dvtx 1\leq n\leq M\}$ jointly)]\label{alg2m}
%
\begin{longlist}[(1)]
\item[(1)] For each index $i$, simulate $M^{i}$ and $(W_{n}^{k}(i)\dvtx n\geq0,
k\geq1, nk<M)$. Compute $M=\max_{i} M^{i}\vee m_{0}$. (As discussed
earlier, $%
M^i$'s are simulated by adapting Algorithm \ref{alg2w}.)

\item[(2)] For each $0\leq n\leq M$ and each index $i$, $\{
W_{n}^{k}(i)\dvtx k<
M^i/n\}$ are already given in step 1. For $k\geq M^i/n$, use
Algorithm \ref{alg2w'}
to simulate $K_n^i$ jointly with $\{W_n^k(i)\dvtx M^i/n\leq k\leq K\}$
conditional on $|W_n^k(i)|\leq4(\sqrt{\log(n+1)}+\sqrt{\log
k})\triangleq
\beta^k>4\sqrt{\log k}$.

\item[(3)] For any $0\leq n\leq M$, compute and output
%
\begin{equation}
\bar{B}_{n,i}^{\varepsilon}(t)= \sum_{i=1}^{d}
A_{ij}\Biggl(\sum_{k=1}^{K_n^{i}}
W_{n}^{k}(i)\int_{0}^{t}
H_{k}(u)\,du\Biggr). \label{bridge}
\end{equation}
\end{longlist}
\end{algorithm}

In step 1 of Algorithm \ref{alg2m}, we can use a similar procedure as in
Algorithm~\ref{alg2w'}
to impose conditioning events of form $|W_n^k(i)|\leq\beta_n^k(i)$ while
simulating $M_i$'s jointly with $W_n^k(i)$'s. In this way, we derive an
algorithm that is able to simulate $M$ jointly with $\{\bar{\mathbf{B}}%
_{n}^{\varepsilon}(\cdot)\dvtx 1\leq n\leq M\}$ conditional on $|W_n^k(i)|\leq
\beta_n^k(i)$ for all $n\geq0$, $k\geq1$ and $1\leq i\leq d$ for any given
sequence of $\{\beta_n^k(i)\}$ such that $\beta_n^k(i)>4(\sqrt{\log
(n+1)}+%
\sqrt{\log k})$.

\renewcommand{\thealgorithm}{2\normalfont{m}$'$}
\begin{algorithm}[(Simulating of $M$ and $\{\bar{\mathbf{B}}%
_{n}^{\varepsilon}(\cdot)\dvtx 1\leq n\leq M\}$ jointly conditional on $%
|W_n^k(i)|\leq\beta_n^k(i)$ for all $n\geq0$, $k\geq1$ and $1\leq
i\leq d$)]\label{alg2m'}
%
\begin{longlist}[(1)]
\item[(1)] For each index $i$, simulate $M_{i}$ and $(W_{n}^{k}(i)\dvtx n\geq0,
k\geq1, nk<M)$ conditional on $|W_n^k(i)|\leq\beta_n^k(i)$ using a similar
procedure as in Algorithm \ref{alg2w'}. Compute $M=\max_{i} M^{i}\vee m_{0}$.

\item[(2)] For each $0\leq n\leq M$ and each index $i$, $\{
W_{n}^{k}(i)\dvtx k<
M^i/n\}$ are already given in step 1. For $k\geq M^i/n$, use Algorithm \ref{alg2w'}
to simulate $K_n^i$ jointly with $\{W_n^k(i)\dvtx M^i/n\leq k\leq K\}$
conditional on $|W_n^k(i)|\leq4(\sqrt{\log(n+1)}+\sqrt{\log k})$. [Note
that $\beta_n^k(i)>4(\sqrt{\log(n+1)}+\sqrt{\log k})>4\sqrt{\log k}$, and
hence this step is well defined.]

\item[(3)] For any $0\leq n\leq M$, compute and output
\[
\bar{B}_{n,i}^{\varepsilon}(t)= \sum_{i=1}^{d}
A_{ij}\Biggl(\sum_{k=1}^{K_n^{i}}
W_{n}^{k}(i)\int_{0}^{t}
H_{k}(u)\,du\Biggr).
\]
\end{longlist}
\end{algorithm}

Algorithm \ref{alg2m'} will be used in the next section in order to keep track of
``conditioning events'' corresponding to condition (2).

\subsubsection{Keeping track of the conditioning events}\label{sec4.3.2}

As we have discussed just prior to the beginning of Section~\ref{sec4.3.1}, we need
to keep track of several conditioning events introduced by conditions
(1) and
(2). First, let us explain how to deal with the conditioning event
corresponding to condition (1). These conditioning events involve only the
random walk $\mathbf{S}(\cdot)$. Now we split $\mathbf{S}(\cdot)$ according
to the sequences of $\{\Gamma_{l}\dvtx l\geq1\}$ and $\{\Delta_{l}\dvtx l\geq1\}$ of
random times defined as follows:
\begin{longlist}[(1)]
\item[(1)] Set $\Delta_{1}=\min\{n\dvtx S_{i}(n)\leq-2m\mbox{ for every }i\}$.

\item[(2)] Define $\Gamma_{l}=\min\{n\geq\Delta
_{l}\dvtx S_{i}(n)>S_{i}(\Delta_{l})+m%
\mbox{ for some }i\}$.

\item[(3)] Put $\Delta_{l+1}=\min\{n\geq\Gamma_{l} I(\Gamma_{l}<\infty
)\vee
\Delta_{l}\dvtx S_{i}(n)<S_{i}( \Delta_{l})-2m\mbox{ for every }i\}$.
\end{longlist}

\begin{figure}

\includegraphics{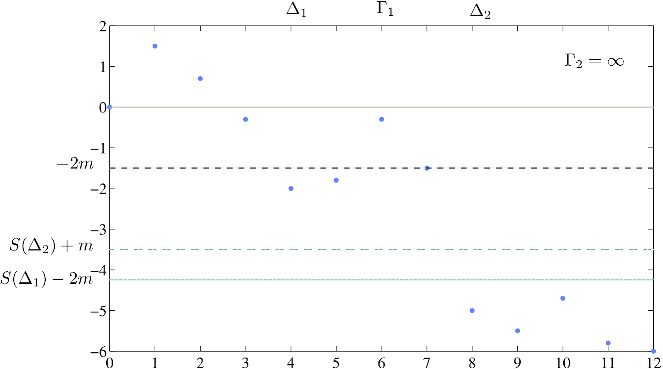}

\caption{Illustration for the random times $\{\Delta_n\}$ and $\{\Gamma_n\}$.}\label{fig1}
\end{figure}


Figure~\ref{fig1} illustrates a sample path of the random walk with the sequence of
random times $\{\Gamma_{l}\dvtx l\geq1\}$ and $\{\Delta_{l}\dvtx l\geq1\}$ in one
dimension. The message is that the joint simulation of $\{\mathbf{S}%
(n)\dvtx n\geq0\}$ with $\{\Gamma_{l}\dvtx l\geq1\}$ and $\{\Delta_{l}\dvtx l\geq1\}$
allows us to keep track of the process $\{\max_{m\geq n}\mathbf
{S}(m)\dvtx n\geq
0\}$, which includes the ``additional information'' introduced by condition
(1). The main steps in the simulation of $\{\mathbf{S}(n)\dvtx n\geq0\}$ jointly
with $\{\Gamma_l\dvtx l\geq1\}$ and $\{\Delta_l\dvtx l\geq1\}$ are explained in
Lemma~2 through Lemma~4 in \citet{BlanchetSigman2011}. The approach of %
\citet{BlanchetSigman2011}, which works in one dimension, could be
modified for multidimensional cases using the change-of-measure as described
in Section~\ref{sec2.3.1}.

Regarding the verification of condition (2) involving $M$ and the
Brownian bridges, as per the discussion in Section~\ref{sec4.3.1}, we just need to
keep track of certain deterministic $\beta_n^k(i)$ for each
$|W_{n}^{k}(i)|$, in order to condition on the events of the form
$|W_n^k(i)|\leq\beta_n^k(i)
$. These events are related to the sequential construction of the
random variable $M$ when testing condition (2) as described in
Section~\ref{sec4.3.1}. Now, we
can write down the integrated version of our algorithm for sampling $%
\tau_\varepsilon$ and $\{\mathbf{Z}^{\varepsilon}(t)\dvtx 0\leq t\leq\tau_{\varepsilon
}\}$
jointly.

\renewcommand{\thealgorithm}{2.1}
\begin{algorithm}[(Simulating $\tau_{\varepsilon}$ and $\{\mathbf{Z}%
^{\varepsilon}(t)\dvtx 0\leq t\leq\tau_{\varepsilon}\}$)]\label{alg2.1}

The output of this algorithm is $\{\mathbf{Z}^{\varepsilon}(t)\dvtx 0\leq
t\leq\tau_{\varepsilon}\}$, and the approximation coalescence time $%
\tau_{\varepsilon}$.
\begin{longlist}[(1)]
\item[(1)] Set $\beta_{n}^{k}(i)=\infty$ for all $n\geq1$, $k\geq1$ and
$1\leq
i\leq d $. Set $L=0$ and $\tau_{\varepsilon}=0$.

\item[(2)] Simulate $\mathbf{S}(n)$ until $\Delta_{l}$, where $l=\min\{
j\dvtx \Gamma
_{j}=\infty,\Delta_{j}>\tau_{\varepsilon}\}$. Compute $\mathbf{Z}^{\varepsilon
}(n)=%
\mathbf{S}(n)-n\bolds{\zeta}$.

\item[(3)] For each $n\in[\tau_{\varepsilon},\Delta_{l}]\cap\mathbb
{Z}_+$ and
each index $1\leq i\leq d$, compute the i.i.d. bridges $\{\bar{\mathbf
{B}%
}^\varepsilon_n(\cdot)\}$ using (\ref{bridge}), in which $K_n^i$ is\vspace*{1pt} jointly
simulated with $(W_n^k(i)\dvtx 1\leq k\leq K_n^i )$ conditional on that $%
|W_n^k(i)|\leq\beta_n^k(i)$ for all $k\geq1$ using Algorithm \ref{alg2w'}.
Given $%
\bar{\mathbf{B}}^\varepsilon_n(\cdot)$ and $\mathbf{S}(n)$ for $n\in[%
\tau_{\varepsilon},\Delta_{l}]\cap\mathbb{Z}_+$, the process $\mathbf
{Z}%
^{\varepsilon}(t)$ for $t\in[\tau_{\varepsilon},\Delta_{l}]$ can be directly
computed. If there exists some $t\geq\Gamma_{l-1}$ such that for all
$t\leq
s\leq\Delta_{l}$, $Z_{i}^{\varepsilon}(t)\geq Z^{\varepsilon}_{i}(s)-2\varepsilon$
and $Z^{\varepsilon}_{i}(t)\geq Z^{\varepsilon}_{i}(\Delta_{l})+m-2\varepsilon
$, set
$\tau_{\varepsilon}\leftarrow t$, and go to step 4. Otherwise, set $%
\tau_{\varepsilon}\leftarrow\Delta_{l}$ and return to step 2.

\item[(4)] Use Algorithm \ref{alg2m'} to simulate $M$ jointly with $(\bar{\mathbf
{B}}%
^\varepsilon_{\tau_{\varepsilon}+n}(\cdot)\dvtx 0\leq n\leq M)$ conditional on $%
|W_{\tau_{\varepsilon}+n}^{k}(i)|\leq\beta_{\tau_\varepsilon+n}^{k}(i)$ for
all $%
n\geq0$, $k\geq1$ and $1\leq i\leq d$. Update $\beta_{\tau_{%
\varepsilon}+n}^{k}(i)\leftarrow4\sqrt{\log{(n+1)}}+4\sqrt{\log{k}}$ for
all $%
n\cdot k\geq M^{i}$. Keep simulating $\mathbf{S}(n)$ until $n=\Delta_l+M$,
and compute $\{\mathbf{Z}^{\varepsilon}(t)\dvtx t\in[\Delta_l,\Delta_l+M]\}$. If
there exist some $t$ and $i$ such that $Z_{i}^{\varepsilon}(t)>Z^{\varepsilon
}_i(%
\tau_{\varepsilon})+\varepsilon$, set $\tau_{\varepsilon} \leftarrow t$ and
return to
step 2.

\item[(5)] Otherwise, stop and output $\tau_{\varepsilon}$ as the approximation
coalescence time along with $(\mathbf{Z}^{\varepsilon}(t)\dvtx 0\leq
t\leq\tau_{\varepsilon})$.
\end{longlist}
\end{algorithm}

\subsection{Computational complexity}
\label{SubSecCCRBM}

In this part, we will discuss the complexity of our algorithm when $d$ and
the other parameters $\bolds{\mu}$ and $A$ are fixed but send the precision
parameter $\varepsilon$ to 0. Denote the total number of random variables
needed by $N(\varepsilon)$ when the precision parameter for the algorithm
is $%
\varepsilon$.

According to Assumption \ref{assD}, the input process $\mathbf{Z}(t)$ equals $-
\bolds{\mu}t+A\mathbf{B}(t)$ with $\mu_{i}>\delta_0>0$. Let $\max_{i,j}
|A_{ij}|=a$. The following result shows that our algorithm's running
time is
polynomial in $1/\varepsilon$:

\begin{theorem}\label{th4}
\label{ThmMain2}Under Assumption \textup{\ref{assD}},
\[
E\bigl[N(\varepsilon)\bigr]=O\biggl(\varepsilon^{-a_C-2}\log\biggl(
\frac{1}{\varepsilon}\biggr)\biggr)\qquad\mbox{as }%
\varepsilon\rightarrow0,
\]
where $a_C$ is a computable constant depending only on $A$.
\end{theorem}

The random variables we need to simulate in the algorithm can be divided
into two parts: first, the random variables used to construct the discrete
random walk $\mathbf{Z}(n)$ for $n\leq T$ and second, the conditional
normals used to bridging between $\mathbf{Z}(n-1)$ and $\mathbf{Z}(n)$.

Since $1(|W|>\eta)$ and $1(|W|\leq\beta)$ are negatively correlated, it
follows that
\[
P\bigl(|W|>\eta| |W|\leq\beta\bigr)\leq P\bigl(|W|>\eta\bigr).
\]
Therefore, the expected number of conditional Gaussian random variables used
for Brownian bridges between $\mathbf{Z}(n-1)$ and $\mathbf{Z}(n)$ is
smaller than the expected number that we would obtain if we use standard
Gaussian random variables instead in steps~3 and 4 in Algorithm
\ref{alg2.1}. Let
$K=\max\{k\dvtx |W_{k}|>\eta_{k}\}\vee K_{0}$ as defined in (\ref
{definiation of
K}). As discussed above, the expected number of truncated Gaussian random
variables needed for each bridge $\bar{B}^{\varepsilon}_{n,i}(\cdot)$ is
bounded by $E[K]$.

Therefore,
\[
E\bigl[N(\varepsilon)\bigr]\leq\bigl(d E[K]+1\bigr) \bigl(E[T]+1\bigr).
\]
To prove Theorem~\ref{th2}, we first need to study $E[K]$ and $E[T]$.

\begin{proposition}\label{pr7}
\[
E[K]=O\biggl(\varepsilon^{-2}\log{ \biggl( \frac{1}{\varepsilon} \biggr) }
\biggr).
\]
\end{proposition}

\begin{pf}
Recall that $\eta_{k}=4\sqrt{\log k}$, and let $p_{k}=P(|W^{k}|>\eta
_{k})$. Then
$p_{k}=O(k^{-4})$. Therefore
\begin{eqnarray*}
E[K] & =&\sum_{n=1}^{\infty}P(K>n)\leq
K_{0}+\sum_{n=K_{0}+1}^{\infty}%
\sum_{k=n}^{\infty}p_{k}
\\
& =&K_{0}+\sum_{k=K_{0}+1}^{\infty}k\cdot
p_{k}\leq K_{0}+O\biggl(\sum_{k=1}%
^{\infty}k^{-3}
\biggr).
\end{eqnarray*}
The second term of the left-hand side is finite and independent of
$\varepsilon$
and $K_{0}$.

On the other side,
\[
\sum_{j=\log_{2}{K_{0}}}2^{-j/2}\sqrt{j+1}\leq
\frac{2}{\log2}(\sqrt {K_{0}%
})^{-1}\biggl(
\sqrt{\log_{2}{K_{0}}}+\frac{2}{\log{2}}\biggr).
\]
Therefore, we can choose $K_{0}=O(\varepsilon^{-2}\log{(\frac{1}{\varepsilon})})$
such that $\sum_{j=\log_{2}{K_{0}}}2^{-j/2}\times\break  \sqrt{j+1}<\varepsilon$.

In order to get the approximation within error at most $\varepsilon$ for
the $d$-dimensional process, according to the Cholesky decomposition as
discussed in
Section~\ref{SubWave}, we should replace $\varepsilon$ by $\frac{\varepsilon}{da}$.
Therefore,
\[
E[K]=O\biggl( \biggl( \frac{\varepsilon}{da} \biggr) ^{-2}\log{ \biggl(
\frac
{da}{\varepsilon
} \biggr) }\biggr)=O\biggl(\varepsilon^{-2}\log{ \biggl(
\frac{1}{\varepsilon} \biggr) }\biggr).
\]
\upqed\end{pf}

What remains is to estimate $E[T]$. Let $T_{a}$ be the time before
the algorithm executes step 4 in a single iteration. Using the same notation
as in Algorithm \ref{alg2.1} and a similar argument as in Section~\ref{sec2.4}, we have
\[
E[T]=\frac{E[T_{a}]+E[T_{m}|T_{m}<\infty]+E[M]}{P(T_{m}<\infty)p},
\]
where
\[
p=P\Bigl(\max_{i}Z_{i}^{\varepsilon}(t)<m+
\varepsilon,\forall 0\leq t\leq M |%
\mathbf{Z}(0)=0;\mathbf{S}(n)<m
\Bigr).
\]
As $\mathbf{Z}^{\varepsilon}(t)=\mathbf{S}(n)-n\zeta\mathbf{1}+A\bar
{\mathbf{B}}_n(t-n)$
and the Brownian bridge $\bar{\mathbf{B}}_n(\cdot)$ is independent of $%
\mathbf{S}(\cdot)$, it follows that
\[
p\geq P\Bigl(\max_{i}\max_{t\geq0}Z_{i}(t)<m|
\mathbf{Z}(0)=0\Bigr).
\]

Since $\mathbf{S}(1)$ is a multidimensional Gaussian random vector with
strictly negative drift, assumptions (C1) to (C3) are satisfied. Applying
Proposition~\ref{pr4}, we can get upper bounds for $E[T_{m}|T_{m}<\infty]$, $%
1/P(T_{m}<\infty)$ and $1/P(\max_{i}\max_{t} Z_{i}(t)<m|\mathbf{Z}(0)=0)$,
which depend only on $d, a$ and $\delta$ and thus are independent of $%
\varepsilon$. Besides, the bound for $E[M]$ can be estimated by the same method
as in Proposition~\ref{pr7} in terms of $\zeta=\delta/2$; hence such a bound is also
independent of $\varepsilon$. Therefore, we only need to estimate $E[T_{a}]$.

\begin{proposition}
\label{PropCorThmMain2}$E[T_{a}]=O(\varepsilon^{-a_{C}})$ as $\varepsilon
\rightarrow0$. Here $a_{C}$ only depends on the matrix $A$. Moreover,
in the
special cases where $A_{ij}\geq0$, $a_{C}=d$.
\end{proposition}

\begin{pf}
Recall that $\mathbf{Z}(t)=-\bolds{\mu}t+A\mathbf{B}(t)$ and $\mu
_i>\delta=2\zeta>0$ as given in Assumption \ref{assD}. We divide the path
of $\mathbf{Z}(t)$ into segments with length $2(m+\varepsilon)/\zeta$,
\[
\biggl\{ \biggl(Z\biggl(k\cdot\frac{2(m+\varepsilon)}{\zeta}+s\biggr)\dvtx 0\leq s\leq
\frac
{2(m+\varepsilon
)}{\zeta} \biggr)\dvtx k\geq0 \biggr\}.
\]
Let
\begin{eqnarray*}
&&N_{b}=\min\biggl\{k\dvtx A\mathbf{B}\biggl(k\cdot
\frac{2(m+\varepsilon)}{\zeta}+s\biggr)-A\mathbf {B}\biggl(k\cdot\frac{2(m+\varepsilon)}{\zeta}\biggr)\leq
\bolds{\varepsilon}\\
&&\hspace*{170pt}\mbox{for all } 0\leq s\leq\frac{2(m+\varepsilon
)}{\zeta} \biggr\}.
\end{eqnarray*}
By independence and stationarity of the increments of Brownian motion, $N_{b}$
is a geometric random variable with parameter
\[
p=P \biggl( A\mathbf{B}(s)\leq \bolds{\varepsilon}\mbox{ for all } 0\leq s\leq
\frac{2(m+\varepsilon
)}{\zeta} \biggr).
\]
On the other hand, since $-\mu_{i}<-2\zeta$, we have:
\begin{longlist}[(1)]

\item[(1)] $Z_{i}(N_{b}\cdot\frac{2(m+\varepsilon)}{\zeta}+s)\leq
Z_{i}(N_{b}%
\cdot\frac{2(m+\varepsilon)}{\zeta})+\varepsilon$, for all $0\leq s\leq
\frac{2(m+\varepsilon)}{\zeta}$.

\item[(2)] $Z_{i}((N_{b}+1)\cdot\frac{2(m+\varepsilon)}{\zeta})\leq
Z_{i}(N_{b}%
\cdot\frac{2(m+\varepsilon)}{\zeta})-m$.
\end{longlist}

Therefore, Algorithm \ref{alg2.1} should execute step 4 after at most $\frac
{2(m+\varepsilon)}{\zeta}(N_{b}+1)$ units of time in a single iteration,
\[
E[T_{a}]\leq\frac{2(m+\varepsilon)}{\zeta}E[N_{b}+1]=
\frac{2(m+\varepsilon)}%
{\zeta}\biggl(1+\frac{1}{p}\biggr).
\]
From this inequality, it is now sufficient to show that $p=O(\varepsilon
^{a_{C}%
})$.

Note that the set $C=\{\mathbf{y}\in\mathbb{R}^{d}\dvtx A\mathbf{y}\leq
\bolds{\varepsilon}\}$ forms a cone with vertex $A^{-1}\bolds
{\varepsilon}$ in
$\mathbb{R}^{d}$ since $A$ is of full rank under Assumption \ref{assD}. Define
$\tau_{C}=\inf\{t\geq0\dvtx \mathbf{B}(t)\notin C\}$ given $\mathbf
{B}(0)=0$, then
\[
p=P\biggl(\tau_{C}>\frac{2(m+\varepsilon)}{\zeta}\biggr).
\]

If $d=2$, it is proved by \citet{Burkholder1977} that $a_{C}=\frac{\pi
}{\theta}$ where $\theta\in{}[0,\pi)$ is the angle formed by the column
vectors of $A^{-1}$. Therefore, we can compute explicitly that
\[
\theta=\arccos \biggl(-\frac{A_{11}A_{21}+A_{12}A_{22}}{\sqrt
{(A_{11}^{2}+A_{12}^{2})(A_{21}
^{2}+A_{22}^{2})}} \biggr),
\]
which only depends on $A$.

On the other hand, if $d\geq3$, applying the results on exit times for
Brownian motions given by Corollary~1.3 in \citet{DeBlassie},
\[
P\biggl(\tau_{C}>\frac{2(m+\varepsilon)}{\zeta}\biggr)\sim u\cdot\bigl\Vert
A^{-1}%
\bolds{\varepsilon}\bigr\Vert^{a_{C}}%
\]
as $\varepsilon\rightarrow0$. Here $\Vert\cdot\Vert$ represent the
Euclidian norm,
and $u$ is some constant independent of $\varepsilon$. The rate $a_{C}$ is
determined by the principal eigenvalue of the Laplace--Beltrami
operator on
$(\mathbf{S}^{d-1}\cap C)$, where $\mathbf{S}^{d-1}$ is a unit sphere centered
at the vertex of $C$, namely $A^{-1}\varepsilon$. The principal eigenvalue only
depends on the geometric features of $C$, and it is independent of
$\varepsilon$;
hence so is $a_{C}$. Since $A$ is given, we have
\[
P\biggl(\tau_{C}>\frac{2(m+\varepsilon)}{\zeta}\biggr)=O\bigl(
\varepsilon^{a_{C}}\bigr)\qquad\mbox{as }\varepsilon\rightarrow0.
\]

Computing $a_{C}$ for $d\geq3$ is not straightforward in general. However,
when $A_{ij}\geq0$, we can estimate $a_{C}$ from first principles.
Indeed, if
$A_{ij}\geq0$ and we let $a=\max A_{ij}$, we have that
\[
C=\bigl\{\mathbf{y}\in\mathbb{R}^{d}\dvtx A\mathbf{y}\leq\bolds{
\varepsilon}%
\bigr\}\subset\biggl\{\mathbf{y}\in\mathbb{R}^{d}
\dvtx y_{i}\leq\frac{\varepsilon}{ad}\biggr\}.
\]
As the coordinates of $\mathbf{B}(t)$ are independent,
\[
p\geq P \biggl( \max_{0\leq t\leq{2(m+\varepsilon)}/{\zeta}}B(t)\leq \frac{\varepsilon}{ad} \biggr)
^{d},
\]
where $B(\cdot)$ is a standard Brownian motion on real line.

Applying the reflection principle, we have
\begin{eqnarray*}
&&P \biggl( \max_{0\leq t\leq{2(m+\varepsilon)}/{\zeta} }B(t)\leq\frac{\varepsilon
}{ad} \biggr) \\
&&\qquad=\int
_{-{\varepsilon}/{(ad)}}^{{\varepsilon}/{(ad)}}\frac{1}%
{\sqrt{2\pi({2(m+\varepsilon)}/{\zeta})}}\exp{ \biggl( -
\frac{x^{2}}%
{2({2(m+\varepsilon)}/{\zeta})} \biggr) }\\
&&\qquad=O(\varepsilon).
\end{eqnarray*}
As a result, $p=O(\varepsilon^{d})$ when the correlations are all nonnegative.
\end{pf}

Given these propositions, we can now prove the main result in this part.

\begin{pf*}{Proof of Theorem~\ref{ThmMain2}}
As we have discussed,
\[
E\bigl[N(\varepsilon)\bigr]\leq\bigl(dE[K]+1\bigr) \bigl(E[\tau_{\varepsilon}]+1
\bigr).
\]
First, by Proposition~\ref{pr7},
$
E[K]=O(\varepsilon^{-2}\log{(\frac{1}{\varepsilon})})$.
Besides, as discussed above,
\[
E[T]\leq\frac{E[T_{a}]+E[T_{m}|T_{m}<\infty]+E[M]}%
{P(T_{m}<\infty)P(\max_{i}\max_{t\geq0} Z_{i}(t)<m|\mathbf{Z}(0)=0)}.
\]
According to Proposition~\ref{PropCorThmMain2}, $E[T_{a}]=O(\varepsilon^{-a_{C}})$, and
$a_{C}$ is a
constant when $A$ is fixed. In the end, as we have discussed,
$E[T_{m}|T_{m}<\infty]$,
$P(T_{m}<\infty)$, $P(\max_{i}\max_{t} Z_{i}(t)<m|\mathbf{Z}(0)=0)$ and $E[M]$
are independent of $\varepsilon$. Therefore,
\[
E[T]=O\bigl(\varepsilon^{-a_{C}}\bigr).
\]
In sum, we have
\[
E\bigl[N(\varepsilon)\bigr]=O\biggl(\varepsilon^{-a_{C}-2}\log{ \biggl(
\frac{1}{\varepsilon} \biggr) }\biggr).
\]
\upqed\end{pf*}

\section{Numerical results}
\label{SectionNumerics}

We first implemented Algorithm \ref{alg1} in order to generate exact samples
from the
steady-state distribution of stochastic fluid networks, and then we implemented
Algorithm \ref{alg2}. Our implementations were performed in Matlab. In all the
experiments we simulated 10,000 independent replications, and we
displayed our
estimates with a margin of error obtained using a 95\% confidence interval
based on the central limit theorem.

For the case of stochastic fluid networks, we considered a 10-station system
in tandem. So, $Q_{i,i+1}=1$ for $i=1,2,\ldots,9$ and $Q_{10,j}=0$ for
all $%
j=1,\ldots,10$. We assume the arrival rate $\lambda=1$ and the job
sizes are
exponentially distributed with unit mean. The service rates $ ( \mu
_{1},\ldots,\mu_{10} ) ^{T}$ are given by $%
(1.55,1.5,1.45,1.4,1.35,1.3,1.25,1.2,1.15,1.1)$. We are interested in
computing the steady-state mean and the second moment of the workload at
each station (i.e., $E[Y_{i} ( \infty ) ]$ and $E[Y_{i} (
\infty ) ^{2}]$ for $i=1,2,\ldots,10$). For a network of this
type, it
turns out that the true values of the quantities we are interested in
can be
computed from the corresponding Laplace transforms as given in %
\citet{Debickietal2007}.

Both the simulation results and the true values are reported in
Table~\ref{tab1}. The
procedure took a few minutes (less than 5) on a desktop, which is quite a
reasonable time.

\begin{table}
\caption{Unbiased estimates of $E[Y_{i} ( \infty ) ]$ and $%
E[Y_{i}^{2} ( \infty ) ]$ for a network with ten stations in
tandem}\label{tab1}
\begin{tabular*}{\textwidth}{@{\extracolsep{\fill}}lcd{1.4}cd{2.4}@{}}
\hline
& \multicolumn{2}{c}{$\bolds{E[Y_{i} ( \infty ) ]}$} & \multicolumn{2}{c@{}}{$\bolds{E[Y_{i}^{2} ( \infty )
]}$} \\[-4pt]
& \multicolumn{2}{c}{\hrulefill} & \multicolumn{2}{c@{}}{\hrulefill} \\
\textbf{Station} & \textbf{Simulation result} &  \multicolumn{1}{c}{\textbf{True value}} &
\textbf{Simulation result}   &\multicolumn{1}{c@{}}{\textbf{True
value}}\\
\hline
\phantom{0}1 & 1.7919${}\pm{}$0.0521&       1.8182 & 10.2755${}\pm{}$0.5289&       10.2479 \\
\phantom{0}2 & 0.1761${}\pm{}$0.0068&       0.1818 & 0.1511${}\pm{}$0.0170&       0.1642 \\
\phantom{0}3 & 0.2171${}\pm{}$0.0083&       0.2222 & 0.2242${}\pm{}$0.0224&       0.2382 \\
\phantom{0}4 & 0.2706${}\pm{}$0.0102&       0.2778 & 0.3462${}\pm{}$0.0339&       0.3610 \\
\phantom{0}5 & 0.3516${}\pm{}$0.0131&       0.3571 & 0.5717${}\pm{}$0.0590&       0.5778 \\
\phantom{0}6 & 0.4737${}\pm{}$0.0171&       0.4762 & 0.9840${}\pm{}$0.0871&       0.9921\\
\phantom{0}7 & 0.6632${}\pm{}$0.0233&       0.6667 & 1.8472${}\pm{}$0.1513&       1.8715\\
\phantom{0}8 & 1.0033${}\pm{}$0.0345&       1.0000 & 4.1004${}\pm{}$0.3377&       4.0300\\
\phantom{0}9 & 1.6497${}\pm{}$0.0542&       1.6667 & 10.3734${}\pm{}$0.7823&       10.6065\\
10 & 3.3200${}\pm{}$0.1040&      3.3333 & 39.2015${}\pm{}$2.9950&       39.3631\\
\hline
\end{tabular*}
\end{table}

We then implemented a two-dimensional RBM example. Let us denote the RBM
by $%
\mathbf{Y}(t)$. The parameters to specify $\mathbf{Y}$ are as follows: drift
vector $\mu=(-1,-1)$, covariance matrix $\Sigma=[1,0;0,1]$ and reflection
matrix $R=[1,-0.2;-0.2,1]$. For this so-call symmetric RBM, one could
compute in close that $E[Y_{1}(\infty)]=E[Y_{2}(\infty)]=5/12\simeq
0.4167$; see, for instance, \citet{DaiHarrison1992}. The output of our
simulation
algorithm is reported in Table~\ref{tab2}.
%
\begin{table}[b]
\caption{Estimates of $E[Y_{i} ( \infty ) ]$ for a 2-dimensional
RBM with precision $\protect\varepsilon=0.01$}\label{tab2}
\begin{tabular*}{\textwidth}{@{\extracolsep{\fill}}lcc@{}}
\hline
& \textbf{Simulation result} & \textbf{True value}\\
\hline
$E[Y_1(\infty)]$ & 0.4164${}\pm{}$0.0137 & 0.4167\\
$E[Y_2(\infty)]$ & 0.4201${}\pm{}$0.0131 & 0.4167\\
\hline
\end{tabular*}
\end{table}

Our implementations here are given with the objective of verifying
empirically the validity of the algorithms proposed. We stress that a direct
implementation of Algorithm \ref{alg2}, although capable of ultimately producing
unbiased estimations of the expectations of RBM, might not be
practical. The
simulations took substantially more time to be produced than those reported
for the stochastic fluid models. This can be explained by the
dependence on $%
\varepsilon$ in Theorem~\ref{ThmMain2}. The bottleneck in the
algorithm is finding a time at which both stations are close to
$\varepsilon
$. An efficient algorithm based on suitably trading a strongly controlled
bias with variance can be used to produce faster running times; we
expect to
report this algorithm in the future.

\section*{Acknowledgments}
The authors thank Offer Kella for pointing out Lem\-ma~\ref{LmK_W_Comp} and
thank Amy Biemiller for her editorial assistance. The authors thank the
Editor and referees for their useful comments and suggestions.

%


%





\printaddresses
\end{document}